\documentclass[final]{siamltex}

\usepackage[utf8]{inputenc}
\usepackage{LatexDefinitions}
\usepackage{LatexDefinitions_Thm}

\hypersetup{  
   pdftitle={Convergence of Entropic Schemes for OT},
   pdfauthor={Carlier, Duval, Peyr\'e, Schmitzer},
   pdfsubject={Convergence of Entropic Schemes for OT}
}

\pdfoutput=1
\renewcommand{\correc}[1]{#1}

\newtheorem{remark}[theorem]{Remark}
\newtheorem{assumption}[theorem]{Assumption}

\newcommand{\qedhere}{} 

\title{Convergence of Entropic Schemes for \\ Optimal Transport and Gradient Flows}

\author{%
Guillaume Carlier\footnote{CEREMADE, Universit\'e Paris-Dauphine, Place du Marechal De Lattre De Tassigny, 75775 Paris 16, FRANCE, \url{{carlier,schmitzer}@ceremade.dauphine.fr} }, \quad
Vincent Duval\footnote{INRIA, MOKAPLAN, Domaine de Voluceau Le Chesnay, FRANCE, \url{vincent.duval@inria.fr}}, \\ %
Gabriel Peyr\'e\footnote{DMA, Ecole Normale Sup\'erieure, 45 rue d'Ulm, F-75230 PARIS cedex 05, FRANCE, \url{gabriel.peyre@ens.fr} }, \quad %
Bernhard Schmitzer$^*$}

\date{\today}

\begin{document}

\maketitle


\begin{abstract}
	Replacing positivity constraints by an entropy barrier is popular to approximate solutions of linear programs. 
	In the special case of the optimal transport problem, this technique dates back to the early work of Schr\"odinger. 
	This approach has recently been used successfully to solve optimal transport related problems in several applied fields such as imaging sciences, machine learning and social sciences.
	The main reason for this success is that, in contrast to linear programming solvers, the resulting algorithms are highly parallelizable and take advantage of the geometry of the computational grid (e.g.\ an image or a triangulated mesh). 
	The first contribution of this article is the proof of the $\Gamma$-convergence of the entropic regularized optimal transport problem towards the Monge-Kantorovich problem for the squared Euclidean norm cost function. 
	This implies in particular the convergence of the optimal entropic regularized transport plan towards an optimal transport plan as the entropy vanishes. 
	Optimal transport distances are also useful to define gradient flows as a limit of implicit Euler steps according to the transportation distance.
	Our second contribution is a proof that implicit steps according to the entropic regularized distance converge towards the original gradient flow when both the step size and the entropic penalty vanish (in some controlled way).
\end{abstract}



\section{Introduction}

\subsection{Motivation}

Optimal transport (OT) offers an elegant solution to many theoretical and practical problems at the interface between probability, partial differential equations and optimization, as highlighted in the monograph of Villani~\cite{Villani03}. This success however comes at a high computational price, since the Kantorovich formulation of OT requires the solution of a linear program over distributions on a product space. Other problems related to OT, such as the computation of OT barycenters~\cite{Carlier_wasserstein_barycenter} and gradient flows~\cite{JKO1998} are even more challenging. 

It is thus of primary interest to find proxies for OT distances, that can provably approximate faithfully the true distance and  transport plan, while offering a better computational complexity than traditional linear programming solvers such as simplex methods~\cite{Burkard09} or interior points methods~\cite{nesterov1994interior}. This article explores the use of an entropic smoothing of the initial linear program, that was proposed initially by Schr\"odinger~\cite{Schroedinger31} (see~\cite{RuschendorfThomsen,LeonardSchroedinger}), and that has recently been revitalized in fields as diverse as machine learning~\cite{CuturiSinkhorn}, computer graphics~\cite{ConvolutionalOT} and social sciences~\cite{Galichon-Entropic}.

\subsection{Related Work}

\paragraph{Entropic regularization of OT}

The use of entropy as a barrier function for positivity constraints has a long tradition in the linear programming literature, see for instance~\cite{CominettiAsympt}. In the specific case of the Kantorovich linear program, this corresponds to Schr\"odinger's problem~\cite{Schroedinger31}, which equivalently reads as the projection of a Gibbs distribution on fixed marginal constraints according to the Kullback-Leibler (KL) divergence, see~\cite{RuschendorfThomsen,LeonardSchroedinger}.

A major interest of this regularized formulation is not really the fact that the entropy is a very good barrier (in the sense of interior points algorithms~\cite{nesterov1994interior}), but rather that its algebraic properties makes it the unique possible choice when it comes to solving iteratively the regularized problem using alternating KL projections~\cite{bregman1967relaxation}. Indeed, projections on the row and column marginal constraints are obtained in closed form by a simple diagonal scaling. These iterative projections exactly correspond to the celebrated Sinkhorn's algorithm~\cite{Sinkhorn64,SinkhornKnopp67,Sinkhorn67}, that was proposed initially under the name of ``IPFP''~\cite{DemingStephanIPFP}. 

The simplicity of this method, as well as the ability to do parallel implementations, make it the method of choice for machine learning applications~\cite{CuturiSinkhorn} where one needs to compute many optimal transports between a large number of input densities. Furthermore, on geometric domains (such as translation-invariant grids or Riemannian manifolds) it is possible to compute in nearly linear time (with respect to the number of sampling points) multiplications by the Gibbs kernel, which makes it even more attractive for applications in imaging and graphics~\cite{ConvolutionalOT}. 

Beyond the initial Kantorovich problem, many more ``transport-like'' linear programs can be tackled by the same approach. One of the most appealing is the computation of Wasserstein barycenters (as defined in~\cite{Carlier_wasserstein_barycenter}), we refer to~\cite{ConvolutionalOT} for an extensive set of illustrations in 2-D, 3-D and on manifolds. 
Generalized transportation problems (such as partial transport) can also be solved this way (see~\cite{BregmanProj2015}), at the expense of replacing the initial Bregman iterative projection method~\cite{bregman1967relaxation} by Dykstra's algorithm (initially defined in~\cite{Dykstra83}) with Bregman divergences~\cite{CensorReich-Dykstra,bauschke-lewis,BregmanCensorReich-Dykstra}.

\paragraph{Entropic regularization of Wasserstein Gradient flows}

It is possible to use the formal Riemannian structure of the Wasserstein space to define gradient flows of energy functionals. This was initiated in the seminal work of Jordan, Kinderlehrer and Otto \cite{JKO1998} who showed that the Fokker-Planck equation can be viewed as the gradient flow of the entropy relative to the equilibrium measure.  The analysis of \cite{JKO1998} strongly relies on computing implicit discrete time-stepping according to the Wasserstein distance (see Section~\ref{sec:JKO-Setup} for the definition of this scheme). 
The time-continuous evolution is obtained by taking the limit of small time-step sizes.
This general theory  of gradient flows in metric spaces has developed very rapidly in recent years, as detailed in the reference textbook~\cite{AmbrosioGradientFlows2005}.
%
%
Since the initial work of \cite{JKO1998}, many other PDEs have been derived as Wasserstein gradient flows for well chosen functionals,  such as the heat equation on manifolds~\cite{ErbarHeatManifold}, the porous medium equation~\cite{otto2001geometry}, more general degenerate parabolic PDEs~\cite{agueh2002existence}, the Keller-Segel equation~\cite{blanchet2008convergence}, higher order PDEs~\cite{GianazzaARMA} (see also~\cite{Burger-JKO} for applications in imaging) and crowd motion~\cite{maury2010macroscopic}.

This implicit formulation has many advantages on the theoretical side. Most notably, it allows to prove existence of solutions to some highly non-linear PDEs, and also to give meaning to minimizing flows of non-smooth functionals (see for instance~\cite{maury2010macroscopic}). On the computational side, the advantages are less clear, since each step requires the solution of a convex program. There is however a recent wave of activity on finding suitable discretizations and numerical solvers that can be both fast and stable. 
Put aside the specific case of 1-D methods~\cite{kinderlehrer1999approximation,blanchet2008convergence,blanchet2012optimal,agueh2013one,Matthes1D}, these methods can be classified as either Eulerian methods with finite difference or finite volumes~\cite{burger2010mixed,CarrilloFiniteVolume} or Lagrangian methods \cite{carrillo2009numerical,Westdickenberg2010,BuddMoving,JDB-JKO}.

We consider here the numerical scheme proposed in~\cite{Peyre-JKO}, that performs an Eulerian discretization of the JKO time-stepping, where one replaces the usual Wasserstein distance by its entropy smoothed approximation. This method leverages the reformulation of this smooth optimization problem as a KL projection, and makes use of Dykstra's algorithm to obtain a fast numerical scheme that scales to large 2-D grids and meshed surfaces. The price to pay is the presence of an extra diffusion that is created by the entropic smoothing. It is the purpose of the present paper to study the discrete gradient flow defined by this scheme, and in particular to analyze the impact of this additional diffusion.


\subsection{Contribution}

Our first contribution (\thref{thm:OTGammaConvergence}) is the proof of $\Gamma$-convergence of the entropy regularized OT functional to the unregularized variant in the limit of vanishing regularization, for the case of the squared Euclidean distance cost. 
This contribution is closely related to the work of Christian L\'eonard~\cite{LeonardSchroedingerMK2012}, that proves $\Gamma$-convergence of various functionals related to this entropic smoothing. 
While the work of L\'eonard is much more general (it covers more general cost functions on more general base spaces), our approach has the advantage of being self-contained, with a short and direct proof that leverages the geometry of the Euclidean cost function. We also show how to extend this convergence result to cover more general variational problems, such as the computation of barycenters (Proposition~\ref{prop-barycenters}). 

Our second contribution (\thref{thm:JKOConvergence}) is the proof of the convergence of the discrete entropic smoothing of the JKO flow, when both the step size and the smoothing strength tend to zero in a suitable joint limit. The smoothing strength must approach zero sufficiently faster than the step size, which corresponds to the requirement that the extra diffusivity of the scheme should vanish in the limit. 
This contribution bears some similarities with~\cite{AdamsCMP} that  studies a similar gradient flow related to the Schr\"odinger problem and \correc{with~\cite{BraidesGamma2014} that discusses $\Gamma$-convergence results in the framework of minimizing movements}.

\subsection{Notation}
Denote by $\probAll(\R^n)$ the Borel probability measures on $\R^n$, by $\prob(\R^n) \subset \probAll(\R^n)$ those with finite second moments (see Eq.~\eqref{eq:Moment}), by $\probacAll(\R^n) \subset \probAll(\R^n)$ those that are absolutely continuous w.r.t.\ the Lebesgue measure (unless otherwise stated, absolute continuity will always be relative to the Lebesgue measure) and by $\probac(\R^n) = \prob(\R^n) \cap \probacAll(\R^n)$ their intersection.
By $\Pi(\mu,\nu) \subset \prob(\R^n \times \R^n)$ we denote the set of transport plans between measures $\mu$, $\nu \in \prob(\R^n)$. This is the set of measures in $\prob(\R^n \times \R^n)$ with $\mu$ and $\nu$ as first and second marginals respectively. By $\PiAc(\mu,\nu) \subset \probac(\R^n \times \R^n)$ denote the set of absolutely continuous transport plans for $\mu$, $\nu \in \probac(\R^n)$. By abuse of notation, for absolutely continuous measures we use the same symbol for the measure and its  density with respect to Lebesgue's measure.

For $R>0$ denote by $B_R$ the open ball in $\R^n$ of radius $R$ centered at the origin.

For a sufficiently smooth function $\phi : \R^n \rightarrow \R$ we denote by $\nabla \phi$, $\Delta \phi$ and $\nabla^2 \phi$ its gradient, Laplacian and Hessian. For a differentiable $\phi : \R^n \rightarrow \R^n$ we denote by $D \phi$ its Jacobian matrix.

\begin{definition}[Some Norms]
	For a vector space $X$ with norm $|\cdot|$ and $\phi \in C(\R^n,X)$, $\psi \in C([a,b] \times \R^n,X)$, $a<b$, denote
	\begin{align*}
		\|\phi\|_\infty & = \sup_{x \in \R^n} |\phi(x)|\,,&
		\|\psi\|_{\infty,\infty} & = \sup_{(t,x) \in [a,b] \times \R^n} |\psi(t,x)|\,.
	\end{align*}
	The relevant interval $[a,b]$ should be clear from the context.
	In this article we consider vector valued functions $X=\R^n$ with the Euclidean norm and matrix valued functions $X= \R^n \times \R^n$ with the spectral norm.
\end{definition}

Let us briefly comment on the topologies and notions of convergence that we use in this article.
\begin{definition}[Topologies and Convergence]
	\label{def:Topologies}
	The \emph{narrow} topology (sometimes also called weak topology) on $\probAll(\R^n)$ is induced by integration against the set $C_b(\R^n)$ of bounded continuous test functions.
	
	A subset $S \subset \probAll(\R^n)$ is called \emph{tight} if for any $\veps > 0$ there is a compact set $K \subset \R^n$ such that $\mu(\R^n \setminus K) \leq \veps$ for all $\mu \in S$. According to \emph{Prokhorov's theorem}, a subset of $\probAll(\R^n)$ is pre-compact for the narrow topology if and only if it is tight (cf.~\cite[Thm.~5.1.3]{AmbrosioGradientFlows2005}).
	
	$\probacAll(\R^n)$ is not narrowly closed: sequences in $\probacAll(\R^n)$ may converge narrowly to Dirac measures $\notin \probacAll(\R^n)$. Therefore we also view $\probacAll(\R^n)$ as a subset of $L^1(\R^n)$ and equip it with the corresponding $L^1$-weak topology. Note that $L^1$-weak convergence implies narrow convergence. To establish convergence results for non-linear PDEs we also take into account the strong $L^1$-topology.
	
	Finally, there is also the topology induced by the Wasserstein distance. The following three are equivalent \cite[Def.~6.8, Thm.~6.9]{Villani-OptimalTransport-09}:
	\begin{itemize}
		\item convergence in the Wasserstein distance,
		\item narrow convergence + convergence of finite second moments,
		\item convergence of integration against all test functions that grow at most as $x \mapsto |x|^2$ at infinity.
	\end{itemize}
\end{definition}



\section{Entropy Regularized Optimal Transport}
\label{sec:OT}

In this section we are investigating the modification of the standard 2-Wasserstein distance on $\prob(\R^n)$ by adding an entropy regularization term on the transport plan.

\subsection{Set-up}
\label{sec:OT-Setup}

\begin{definition}[Entropy]
\label{def:Entropy}
For $\mu \in \probAll(\R^m)$ we define the (negative) entropy
\begin{align}
	\label{eq:Entropy}
	H_{\R^m}(\mu) & \eqdef
		\begin{cases}
			\int_{\R^m} \mu(x)\,\log(\mu(x))\,dx & \text{for } \mu \in \probacAll(\R^m)\,, \\
			+ \infty & \text{otherwise.}
		\end{cases}
\end{align}
For $\mu \in \probacAll(\R^m)$ we will also require the `negative' and `positive' part of the entropy:
\begin{align*}
	\HRkneg{m}(\mu) \eqdef \int_{\R^m} | \min\{\mu(x)\,\log(\mu(x)),0\} |\,dx\,,\,\,
	\HRkpos{m}(\mu) \eqdef \int_{\R^m} \max\{\mu(x)\,\log(\mu(x)),0\}\,dx\,.
\end{align*}
We simply write $H_1$, $\Hneg$, $\Hpos$ for $H_{\R^n}$, $\HRkneg{n}$, $\HRkpos{n}$ and analogous for 2 and $\R^n \times \R^n$.
By virtue of \thref{cor:fctLSC} the functions $H_1$ and $H_2$ are l.s.c.~w.r.t.~the narrow topology under bounded second moments.
\end{definition}

\noindent We will use many estimates involving the second moment of a measure $\mu \in \probAll(\R^m)$:
\begin{align}
	\label{eq:Moment}
	M_{\R^m}(\mu) & \eqdef \int_{\R^m} |x|^2 d\mu(x).
\end{align}
The set of measures with bounded second moments is given by $\prob(\R^m) = \{ \mu \in \probAll(\R^m) : M_{\R^m}(\mu) < \infty \}$. We often just write $M(\mu) = M_{\R^n}(\mu)$ for $\mu \in \probAll(\R^n)$.

\begin{definition}[2-Wasserstein distance on $\prob(\R^n)$]
Let $c : \R^n \times \R^n \rightarrow \R$, $c(x,y) = |x-y|^2$. For $\gamma \in \prob(\R^n \times \R^n)$ we write
\begin{align}
	(c,\gamma) & \eqdef \int c(x,y)\,d\gamma(x,y) = \int |x-y|^2\,d\gamma(x,y)\,.
\intertext{The 2-Wasserstein distance on $\prob(\R^n)$ is then given by}
	\label{eq:W2}
	W^2(\mu,\nu) & \eqdef \inf_{\gamma \in \Pi(\mu,\nu)} (c,\gamma).
\intertext{It is well known that $W$ defines a distance on $\prob(\R^n)$ \cite{Villani03}.
In this article we consider an entropy regularized variant of optimization problem \eqref{eq:W2}:}
	\label{eq:W2Entropy}
	W^2_\veps(\mu,\nu) & \eqdef \inf_{\gamma \in \Pi(\mu,\nu)} \left( \vphantom{\sum} (c,\gamma) + \veps\,H_2(\gamma) \right)
\end{align}
where $\veps>0$ is the regularization parameter.
\end{definition}
We will investigate how $W^2_\veps$ can be used to approximate $W^2$. We start by establishing $\Gamma$-convergence of the entropy regularized version towards the unregularized functional as regularization decreases. In Section \ref{sec:JKO} we then investigate replacing $W^2$ by $W^2_\veps$ in a time-discrete gradient flow scheme.

Note that from the inequality $|y|^2 \leq 2|x|^2 + 2|x-y|^2$ for $x$, $y \in \R^n$ we find immediately
\begin{align}
	\label{eq:W2MomentBound}
	M(\nu) \leq 2\,M(\mu) + 2\,W^2(\mu,\nu)
\end{align}
for $\mu$, $\nu \in \prob(\R^n)$.

\subsection{Preliminary Results}

We first show some simple properties of $W^2_\veps$.

\begin{proposition}[Existence of Minimizing Couplings for $W^2_\veps(\mu,\nu)$]
\thlabel{thm:W2EpsExistence}
Let $\mu, \nu \in \probac(\R^n)$ such that $H_1(\mu)$, $H_1(\nu) < \infty$. Then there is a $\gamma \in \PiAc(\mu,\nu)$ with $H_2(\gamma) < \infty$ that attains the infimum in \eqref{eq:W2Entropy}.
\end{proposition}
\begin{proof}
$\Pi(\mu,\nu)$ is tight \cite[Lemma 4.4]{Villani-OptimalTransport-09}, narrowly closed and hence narrowly compact. Consequently, let $(\gamma_k)_{k \in \N}$ be a minimizing sequence in $\Pi(\mu,\nu)$ narrowly converging to some $\gamma \in \Pi(\mu,\nu)$.
Narrow lower semi-continuity of $\gamma \mapsto (c,\gamma)$ is shown in \cite[Lemma 4.3]{Villani-OptimalTransport-09}. \correc{As for $H_2$, observe that $\gamma_k\in\Pi(\mu,\nu)$ implies that $M_{\R^n\times \R^n}(\gamma_k)=M_{\R^n}(\mu)+M_{\R^n}(\nu)<+\infty$, and~\thref{cor:fctLSC} ensures that $H_2$ is l.s.c. (see Def.~\ref{def:Entropy})}. Hence, the limit $\gamma$ is indeed a minimizer.
Since $\mu$, $\nu$ have finite second moments and entropy, the infimum is $< \infty$ (the product measure $\mu \otimes \nu$ is \correc{admissible} with finite cost). $c$ is non-negative, i.e.\ so is $(c,\gamma)$.
Therefore $H_2(\gamma)<\infty$ and thus $\gamma \in \PiAc(\mu,\nu)$.
\end{proof}

\begin{lemma}[Narrow Lower Semi-continuity of $W_\veps$ under Bounded Entropy and Moments]
	\thlabel{thm:W2Epslsc}
	Let $(\mu_k)_{k \in \N}$, $(\nu_k)_{k \in \N}$ be sequences in $\probac(\R^n)$ converging narrowly to some $\mu$, $\nu \in \probac(\R^n)$ respectively. Let further $H_1$ and $M$ of $(\mu_k)_k$, $(\nu_k)_k$, $\mu$ and $\nu$ be uniformly bounded by some $C < \infty$. Then
	\begin{equation*}
		W^2_\veps(\mu,\nu) \leq \liminf_{k \rightarrow \infty} W^2_\veps(\mu_k,\nu_k)\,.
	\end{equation*}
\end{lemma}
\begin{proof}
  \correc{Up to the extraction of a (not relabeled) subsequence, we may assume that $\lim_{k \rightarrow \infty} W^2_\veps(\mu_k,\nu_k)=\liminf_{k \rightarrow \infty} W^2_\veps(\mu_k,\nu_k)$.  Now, let }
 $(\gamma_k)_{k \in \N}$ be a sequence of optimal transport plans for $W^2_\veps(\mu_k,\nu_k)$, taking values in $\probac(\R^n \times \R^n)$ (\thref{thm:W2EpsExistence}).
  Since $(\mu_k)_{k \in \N}$ and $(\nu_k)_{k \in \N}$ are convergent sequences (or have bounded second moments), the sets $\{\mu_k\}_{k \in \N}$ and $\{\nu_k\}_{k \in \N}$ are tight, hence so is $\{\gamma_k\}_{k \in \N}$ \cite[Lemma~4.4]{Villani-OptimalTransport-09}. Consequently, up to extracting a subsequence, $(\gamma_k)_k$ converges narrowly to some $\gamma \in \prob(\R^n \times \R^n)$.
	Taking the relation $\gamma_k \in \Pi(\mu_k,\nu_k)$ to the limit we find $\gamma \in \Pi(\mu,\nu)$.
	By narrow l.s.c.~of $\gamma \mapsto (c,\gamma)$ and $H_2$ (c.f.~\thref{thm:W2EpsExistence}) the claim follows.
\end{proof}

\begin{lemma}
	\thlabel{thm:W2EpsStrictConvexity}
	For a fixed $\mu \in \probac(\R^n)$ the map
	\begin{equation*}
		\probac(\R^n) \ni \nu \mapsto W^2_\veps(\mu,\nu)
	\end{equation*}
	is strictly convex.
\end{lemma}
\begin{proof}
	For two $\nu_1$, $\nu_2 \in \probac(\R^n)$ let $\gamma_1$, $\gamma_2$ be the corresponding optimal couplings in $W^2_\veps$. For $0 < \alpha < 1$ let $\nu = \alpha\,\nu_1 + (1-\alpha)\,\nu_2$. Then $\gamma = \alpha\,\gamma_1 + (1-\alpha)\,\gamma_2 \in \PiAc(\mu,\nu)$. By linearity of $(c,\cdot)$ and strict convexity of $H_2$, the cost of $\gamma$ (an upper bound for $W^2_\veps(\mu,\nu)$) is strictly smaller than $\alpha\,W^2_\veps(\mu,\nu_1) + (1-\alpha) W^2_\veps(\mu,\nu_2)$.
\end{proof}

\subsection[Gamma-convergence of Entropy Regularized Optimal Transport]{$\Gamma$-convergence of Entropy Regularized Optimal Transport}

\begin{definition}
	\label{def:OTConvergenceSetup}
	Throughout this Section we consider the following set-up: let $\mu$, $\nu \in \probac(\R^n)$ with finite entropy and let $(\veps_k)_{k \in \N}$ be a non-negative sequence converging to zero. We introduce the sequence of functionals
	\begin{align}
		\mc{F}_k : {} & \probAll(\R^n \times \R^n) \rightarrow \R \cup \{\infty\}, \qquad &
		\gamma \mapsto & \begin{cases}
			(c,\gamma) + \veps_k\,H_2(\gamma) & \tn{if } \gamma \in \Pi(\mu,\nu), \\
			\infty & \tn{else,} \end{cases} \\
		\intertext{and}
		\mc{F} : {} & \probAll(\R^n \times \R^n) \rightarrow \R \cup \{\infty\}, \qquad &
		\gamma \mapsto & \begin{cases}
			(c,\gamma) & \tn{if } \gamma \in \Pi(\mu,\nu), \\
			\infty & \tn{else.} \end{cases}
	\end{align}
\end{definition}

\noindent \correc{The aim of this section is to prove the $\Gamma$-convergence of the entropy-regularized functional $\mc{F}_k$ towards $\mc{F}$. We say that $(\mc{F}_k)_{k\in \N}$ $\Gamma$-converges towards $\mc{F}$ w.r.t the narrow topology if the following two conditions hold for every $\gamma\in \probAll(\R^n \times \R^n)$:
  \begin{itemize}
    \item (Liminf Condition) For any sequence $(\gamma_k)_{k \in \N}$ in $\probAll(\R^n \times \R^n)$ such that $\gamma_k \rightarrow \gamma$ narrowly,
	\begin{align}
		\mc{F}(\gamma) \leq \liminf_{k \rightarrow \infty} \mc{F}_k(\gamma_k)\,,
	\end{align}
  \item (Limsup Condition) There exists a (recovery) sequence $(\gamma_k)_{k\in \N}$ such that $\gamma_k \rightarrow \gamma$ narrowly and 
    \begin{align}
		\mc{F}(\gamma) \geq \limsup_{k \rightarrow \infty} \mc{F}_k(\gamma_k)\,.
    \end{align}
  \end{itemize}
We refer the reader to~\cite{Dalmaso1993} for more details about $\Gamma$-convergence.
Let us simply note a crucial property of that notion: if $(\mc{F}_k)_{k\in\N}$ is equi-coercive and $\Gamma$-converges towards $\mc{F}$, then $\lim_{k\to+\infty}\inf {\mc{F}_k}=\inf\mc{F}$ by~\cite[Theorem 7.8]{Dalmaso1993}. Moreover if $(\gamma_k)_{k\in\N}$ is a sequence of minimizers of $\mc{F}_k$ for each $k\in \N$, then any cluster point of $(\gamma_k)_{k\in\N}$ is a minimizer of $\mc{F}$ (see~\cite[Proposition 7.18]{Dalmaso1993}).
In particular, if the minimizer of $\mc{F}$ is unique, then the whole sequence $(\gamma_k)_{k\in\N}$ converges towards this minimizer.

That is precisely the case in our setting. First, for any $k\in \N$, all admissible $\gamma$ should belong to $\Pi(\mu,\nu)$ which is tight, hence the equicoercivity. Since we assume that $H_1(\mu)<+\infty$, $\mu$ does not give mass to small sets and by~\cite[Theorem 2.12]{Villani03} the optimal transport plan $\gamma$ for $W^2(\mu,\nu)$ is unique. As a result, we shall obtain 
the convergence of the value of $W^2_{\veps_k}(\mu,\nu)$ toward $W^2(\mu,\nu)$ when $k\to+\infty$, and the convergence of the optimal transport plan $\gamma_{\varepsilon_k}$ solving~\eqref{eq:W2Entropy} towards an optimal transport plan for~\eqref{eq:W2}. More precisely,

\begin{theorem}
	\thlabel{thm:OTGammaConvergence}
  The sequence  $(\mc{F}_k)_{k\in \N}$ $\Gamma$-converges to $\mc{F}$ w.r.t.~the narrow topology. As a result, 
  \begin{equation*}
    \lim_{k\to+\infty} W^2_{\varepsilon_k}(\mu, \nu)=W^2(\mu,\nu),
  \end{equation*}
  and if $\gamma_k$, $k\in \N$, (resp. $\gamma$) denotes the unique optimal transport plan for $W^2_{\varepsilon_k}(\mu, \nu)$ (resp. $W^2(\mu,\nu)$), then $\gamma_k\to \gamma$ in the narrow topology.
\end{theorem}

The proof of~\thref{thm:OTGammaConvergence} is divided in~\thref{thm:OTGammaLiminf} and \thref{thm:OTGammaLimsup}.
}

\begin{proposition}[Liminf Condition]
	\thlabel{thm:OTGammaLiminf}
	Given the set-up of Definition \ref{def:OTConvergenceSetup}, let $\gamma \in \probAll(\R^n \times \R^n)$ and a sequence $(\gamma_k)_{k \in \N}$ in $\probAll(\R^n \times \R^n)$ such that $\gamma_k \rightarrow \gamma$ narrowly. Then
	\begin{align}
		\mc{F}(\gamma) \leq \liminf_{k \rightarrow \infty} \mc{F}_k(\gamma_k)\,.
	\end{align}
\end{proposition}
\begin{proof}
	If $\gamma \notin \Pi(\mu,\nu)$ then the condition holds trivially: since $\Pi(\mu,\nu)$ is closed, in this case $\gamma_k \notin \Pi(\mu,\nu)$ for sufficiently large $k$.
	
	If $\gamma \in \Pi(\mu,\nu)$ it suffices to consider sub-sequences of $(\gamma_k)_k$ with values in $\PiAc(\mu,\nu)$ (as otherwise again, the condition holds trivially). Since $M_{\R^n \times \R^n}(\gamma_k) = M_{\R^n}(\mu) + M_{\R^n}(\nu) < \infty$ there is a finite constant $C<\infty$ such that $\HnegB(\gamma_k) > -C$ (\thref{thm:EntropyMomentBound}) and consequently $\liminf_{k \rightarrow \infty} \veps_k\,H_2(\gamma_k) \geq 0$. The statement then follows from narrow l.s.c.~of $\gamma \mapsto (c,\gamma)$ (cf.~\thref{thm:W2EpsExistence}).
\end{proof}
%


The limsup condition of $\Gamma$-convergence requires considerably more effort and is divided into multiple steps. Note however that the approach we use in the following is still rather short and direct compared to the elaborate framework established in \cite{LeonardSchroedingerMK2012}.

\begin{definition}[Block Approximation]
Let $\mu$, $\nu \in \probac(\R^n)$, $H_1(\mu)$, $H_1(\nu)< \infty$, $\gamma \in \Pi(\mu,\nu)$.
Given $k=(k_1,\ldots,k_n) \in \Z^n$ we define $Q_k \eqdef [k_1,k_1+1) \times \ldots \times [k_n,k_n+1) \subset \R^n$, and for $\ell>0$ we write $Q_k^\ell = \ell \cdot Q_k$.
We define the \emph{block approximation} of $\gamma$ at scale $\ell$ by
\begin{align}
	\gamma_\ell & \eqdef \sum_{(j,k)\in(\Z^n)^2} \gamma\left(\Qjl\times \Qkl\right)(\mujl\otimes \nukl),
\end{align}
where for every Borel set $\sigma \subset \R^n$,
\begin{align}
	\mujl(\sigma) & \eqdef \begin{cases}
		\frac{\mu(\sigma \cap \Qjl)}{\mu(\Qjl)} & \text{if } \mu(\Qjl)>0,\\
		0 & \text{otherwise.}
	\end{cases}
	& \tn{ and } \quad
	\nukl(\sigma) & \eqdef \begin{cases}
		\frac{\nu(\sigma \cap \Qjl)}{\nu(\Qkl)} & \text{if } \nu(\Qkl)>0,\\
		0 & \text{otherwise.}
	\end{cases}
\end{align}
It can be shown that $\gamma_\ell$ is indeed a Borel probability measure with density
\begin{align}
	\gamma_\ell(x,y) = \begin{cases}
		\gamma(\Qjl\times\Qkl) \frac{\mu(x)\,\nu(y)}{\mu(\Qjl)\,\nu(\Qkl)}
			& \text{if } \mu(\Qjl)>0 \tn{ and } \nu(\Qkl)>0,\\
		0 & \text{otherwise,}
	\end{cases}
\end{align}
where $(j,k)\in(\Z^n)^2$ is uniquely determined by $(x,y)\in\Qjl\times\Qkl$.
\end{definition}
\noindent The block approximation is illustrated in Fig.~\ref{fig:BlockApproximation}.
\begin{figure}
	\centering
	\includegraphics[]{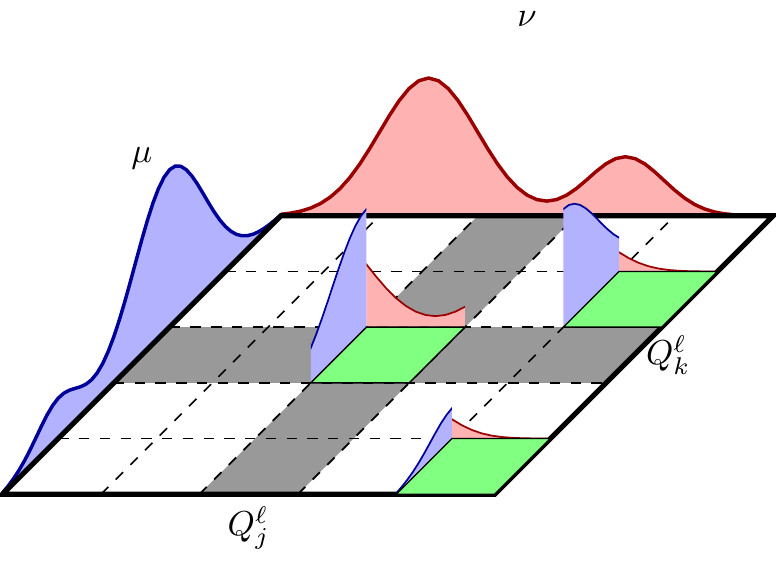}
	\caption{Block approximation of $\gamma$: $\gamma_\ell$ is an approximation of $\gamma$ by pieces of the product measure $\mu \otimes \nu$ such that the mass of $\gamma$ and $\gamma_\ell$ is equal on each cube $\Qkl \times \Qjl$.}
	\label{fig:BlockApproximation}
\end{figure}

\begin{proposition}
	\thlabel{thm:OTConvergenceBlockMarginals}
	For $\gamma \in \Pi(\mu,\nu)$ the block approximation $\gamma_\ell$ is also in $\Pi(\mu,\nu)$.
\end{proposition}
\begin{proof}
	For any Borel set $\sigma \subset \R^n$ we have $(\mujl\otimes \nukl)(\R^n\times \sigma)= \nukl(\sigma)$ if $\mu(\Qjl)>0$ and $0$ otherwise. Thus
	\begin{align*}
		\gamma_\ell(\R^n\times \sigma)&= \sum_{(j,k) \in(\Z^n)^2}\gamma\left(\Qjl\times \Qkl\right)(\mujl\otimes \nukl) (\R^n\times \sigma)
		= \sum_{\substack{(j,k) \in(\Z^n)^2:\\ \mu(\Qjl)>0}}\gamma\left(\Qjl\times \Qkl\right)\nukl(\sigma)\\
		&= \sum_{k\in\Z^n} \nukl(\sigma) \underbrace{\sum_{j\in\Z^n}\gamma\left(\Qjl\times \Qkl\right)}_{=\nu(\Qkl)}
		= \nu(\sigma),
	\end{align*}
where in the third equality, we have used that $\gamma\left(\Qjl\times \Qkl\right)=0$ if $\mu(\Qjl)=0$. Similarly, $\gamma_\ell(\sigma \times \R^n)= \mu(\sigma)$.
\end{proof}

The following Lemma will be useful for studying the limit $\ell \rightarrow 0^+$ of the block approximation.
\begin{lemma}
	\thlabel{thm:OTpartitionConvergence}
	Let $\{Q_i\}_{i \in I}$ be a countable partition of $\R^n$ into Borel sets with $\sup_{i \in I} \diam(Q_i) \leq C < \infty$, i.e.\ $|x-y|^2 \leq C^2$ for $x$, $y \in Q_i$ for any $i \in I$. Let $\mu$, $\nu \in \prob(\R^n)$ such that $\mu(Q_i) = \nu(Q_i)$ for all $i \in I$.
	Then $W^2(\mu,\nu) \leq C^2$.
\end{lemma}
\begin{proof}
	Denote by $\hat{I}$ the subset of $I$ such that $\mu(Q_i)=\nu(Q_i)>0$ for $i \in \hat{I}$. For $i \in \hat{I}$ and every Borel $\sigma \subset \R^n$ let
	\begin{align*}
		\mu_i(\sigma) & = \frac{\mu(\sigma \cap Q_i)}{\mu(Q_i)}
	\end{align*}
	and analogously define $\nu_i$. Clearly all $\mu_i$, $\nu_i \in \prob(\R^n)$, with support contained in $\ol{Q_i}$. For every $i \in \hat{I}$ let $\gamma_i \in \Pi(\mu_i,\nu_i)$ and have $\spt \gamma_i \subset \ol{Q_i}^2$ and thus
	\begin{align*}
		(c,\gamma_i) & = \int_{R^n \times \R^n} |x-y|^2\,d\gamma_i(x,y) = \int_{Q_i \times Q_i} |x-y|^2\,d\gamma_i(x,y) \leq C^2\,.
	\end{align*}
	One finds $\gamma = \sum_{i \in \hat{I}} \mu(Q_i)\,\gamma_i \in \Pi(\mu,\nu)$ and consequently
	\begin{align*}
		W^2(\mu,\nu) & \leq (c,\gamma) = \sum_{i \in \hat{I}} \mu(Q_i)\,(c,\gamma_i) \leq \sum_{i \in \hat{I}} \mu(Q_i)\, C^2 = C^2\,. \qedhere
	\end{align*}
\end{proof}

\begin{corollary}
	\thlabel{thm:OTConvergenceBlockConvergence}
	Let $\mu$, $\nu \in \probac(\R^n)$. For $\gamma \in \Pi(\mu,\nu)$ and its block approximation $\gamma_\ell$, we have
	\begin{align}
		\label{eq:OTgammaBlockConvergence}
		W^2(\gamma,\gamma_\ell) \leq 2\,n\,\ell^2
	\end{align}
	and $\gamma_\ell \rightarrow \gamma$ narrowly as $\ell \rightarrow 0^+$.
\end{corollary}
\begin{proof}
	Inequality \eqref{eq:OTgammaBlockConvergence} is a direct application of \thref{thm:OTpartitionConvergence}.
	This implies narrow convergence (see Def.~\ref{def:Topologies}).
\end{proof}

In this corollary, $W$ denotes the Wasserstein distance over $\R^{n \times n}$. Now we use convergence of $\gamma_\ell$ in the Wasserstein space over $\R^{n \times n}$ to show convergence of the transport cost induced by $\gamma_\ell$ as a coupling for the Wasserstein space over $\R^n$.
\begin{corollary}
	\thlabel{thm:BlockTransportCostConvergence}
	The transport cost of the block approximation converges:
	\begin{align}
		\lim_{\ell\to 0^+} (c,\gamma_\ell) =(c,\gamma)\,.
	\end{align}
\end{corollary}
\begin{proof}
	The function $c : \R^{n \times n} \rightarrow \R$, $(x,y) \mapsto |x-y|^2$ is bounded by $2(|x|^2 + |y|^2)$. Therefore, $W^2(\gamma_\ell,\gamma) \rightarrow 0$ implies convergence of integration w.r.t. $c$ (see Def.~\ref{def:Topologies}).
\end{proof}
%
%
%

\begin{proposition}[Bounding Entropy of Block Approximation]
	\thlabel{thm:OTConvergenceBlockEntropyBound}
	There are constants $C>0$ and $\alpha \in (0,1)$ such that the entropy of the block approximation $\gamma_\ell$ of $\gamma \in \Pi(\mu,\nu)$ at scale $\ell>0$ is bounded by
	\begin{align}
    H_2(\gamma_\ell) \leq H_1(\mu) + H_1(\nu) + C \left( \vphantom{\sum} (M(\mu) + n\,\correc{\ell^2} + 1)^\alpha + (M(\nu)+n\,\correc{\ell^2} + 1)^\alpha - 2\,n\,\log(\ell) \right)\,.
	\end{align}
\end{proposition}

\begin{proof}
	We bound the entropy of $\gamma_\ell$ by 
\begin{align*}
	H_2(\gamma_\ell) & = \sum_{
		\substack{(j,k)\in(\Z^n)^2:\\
			\mu(\Qjl)>0,\,\nu(\Qkl)>0}}
		\int_{\Qjl\times \Qkl} \gamma(\Qjl \times \Qkl) \frac{\mu(x)\,\nu(y)}{\mu(\Qjl)\,\nu(\Qkl)}
			\log\left(\gamma(\Qjl \times \Qkl)\frac{\mu(x)\,\nu(y)}{\mu(\Qjl)\,\nu(\Qkl)}\right)
			dx\,dy \\
	& = \sum_{\substack{(j,k)\in(\Z^n)^2:\\\mu(\Qjl)>0,\,\nu(\Qkl)>0}}
		\gamma(\Qjl \times \Qkl) \left[
		\underbrace{\log\left(\gamma(\Qjl\times\Qkl)\right)}_{\leq 0}
		+ \int_{\Qjl} \frac{\mu(x)}{\mu(\Qjl)} \log\left(\frac{\mu(x)}{\mu(\Qjl)}\right) dx \right. \\ 
	& \qquad \qquad \left.
		+ \int_{\Qkl} \frac{\nu(y)}{\nu(\Qkl)} \log\left(\frac{\nu(x)}{\nu(\Qkl)}\right) dy \right] \\
	& \leq H_1(\mu) + H_1(\nu) - \sum_{j\in\Z^n} \mu(\Qjl) \log\left(\mu(\Qjl)\right)
		- \sum_{k\in\Z^n} \nu(\Qkl)\log\left(\nu(\Qkl)\right).
\end{align*}
The result follows from using the subsequent \thref{thm:DiscreteEntropyBound} on the last two terms.
\end{proof}

\begin{lemma}
	\thlabel{thm:DiscreteEntropyBound}
	For there exist constants $C > 0$, $\alpha \in (0,1)$ such that for $\mu \in \prob(\R^n)$
	\begin{equation}
		\sum_{j\in\Z^n} \mu(\Qjl) \log\left(\mu(\Qjl)\right) \geq -C (M(\mu) + n\,\ell^2 + 1)^\alpha + n\,\log(\ell)\,.
	\end{equation}
\end{lemma}
\begin{proof}
	Consider the density $\mu_\ell(x) = \mu(\Qjl)\,\ell^{-n}$, $j$ uniquely determined by $x \in \Qjl$. Clearly $\mu_\ell(\Qjl) = \mu(\Qjl)$. We have
	\begin{align*}
		\sum_{j\in\Z^n} \mu(\Qjl) \log\left(\mu(\Qjl)\right) & = \int_{\R^n} \mu_\ell(x) \log\left(\mu_\ell(x) \cdot \ell^n\right) \\
		& = H_1(\mu_\ell) + n\,\log(\ell) \geq -C (M(\mu_\ell) + 1)^\alpha + n\,\log(\ell)
	\end{align*}
	for suitable constants $C>0$, $\alpha \in (0,1)$ by virtue of \thref{thm:EntropyMomentBound}.
	Since $\diam(\Qjl)=\sqrt{n}\,\ell$, for all $j \in \Z^n$ have $W^2(\mu,\mu_\ell) \leq n\,\ell^2$ by \thref{thm:OTpartitionConvergence} and thus the conclusion follows from \eqref{eq:W2MomentBound} (with an adjusted constant $C$).
\end{proof}

We now summarize the previous results:
\begin{proposition}[Limsup Condition]
	\thlabel{thm:OTGammaLimsup}
	Given the set-up of Definition \ref{def:OTConvergenceSetup}, for every $\gamma \in \probAll(\R^n \times \R^n)$ there is a non-negative sequence $(\ell_k)_{k \in \N}$ converging to zero, such that
	\begin{align}
		\mc{F}(\gamma) \geq \limsup_{k \rightarrow \infty} \mc{F}_{k}(\gamma_{\ell_k})\,.
	\end{align}
\end{proposition}
\begin{proof}
Note first that for any non-negative sequence $(\ell_k)_{k}$ converging to zero, $\gamma_{\ell_k} \rightarrow \gamma$ narrowly (\thref{thm:OTConvergenceBlockConvergence}).
Again, for $\gamma \notin \Pi(\mu,\nu)$ the condition holds trivially. So in the following let $\gamma \in \Pi(\mu,\nu)$ and consequently $\gamma_{\ell_k} \in \Pi(\mu,\nu)$ (\thref{thm:OTConvergenceBlockMarginals}).

By virtue of \thref{thm:OTConvergenceBlockEntropyBound}, we have
\begin{align*}
	H_2(\gamma_{\ell_k}) \leq H_1(\mu) + H_1(\nu) + C \left( \vphantom{\sum} (M(\mu) + n\,\correc{\ell_k^2} + 1)^\alpha + (M(\nu)+n\, \correc{\ell_k^2} + 1)^\alpha - 2\,n\,\log(\ell_k) \right)
\end{align*}
and thus for $\ell_k=\veps_k$ (implying $\veps_k\,\log(\ell_k) \to 0$ as $k \to \infty$) we have $\correc{\limsup_{k\to +\infty}\veps_k\,H_2(\gamma_{\ell_k})\leq 0}$. The claim then follows from convergence of the transport cost (\thref{thm:BlockTransportCostConvergence}).
\end{proof}


\subsection{Application to Wasserstein Barycenters}

Entropy regularization can also be applied to obtain approximate solutions to more complicated optimization problems involving optimal transport. As an example we sketch application to Wasserstein barycenters, as initially defined in~\cite{Carlier_wasserstein_barycenter} in the un-regularized case.

\begin{definition}[Barycenter Problem]
	Let $(\mu_i)_{i=1}^N \in \prob(\R^n)^N$ be a tuple of marginals with finite entropy and let $(\veps_k)_{k \in \N}$ be a non-negative sequence converging to zero. Denote by
	\begin{align}
		\Pi((\mu_i)_i) = \left\{ (\gamma_i)_{i=1}^N \in \prob(\R^n \times \R^n)^N \,\colon
			\exists\, \rho \in \prob(\R^n) \tn{ such that } \gamma_i \in \Pi(\mu_i,\rho) \tn{ for } i=1,\ldots,N \right\}
	\end{align}
	the set of couplings from $(\mu_i)_i$ to a common second marginal (which is not fixed). Analogous to Definition \ref{def:OTConvergenceSetup} introduce
	\begin{align}
		\mc{G}_k: & \prob(\R^n \times \R^n)^N \mapsto \R \cup \{\infty\}, &
			(\gamma_i)_i & \mapsto \begin{cases}
				\sum_{i=1}^N (c,\gamma_i) + \veps_k \, H_2(\gamma_i) & \tn{if } (\gamma_i)_i \in \Pi((\mu_i)_i) \\
				\infty & \tn{else,}
				\end{cases} \\
		\intertext{and}
		\mc{G}: & \prob(\R^n \times \R^n)^N \mapsto \R \cup \{\infty\}, &
			(\gamma_i)_i & \mapsto \begin{cases}
				\sum_{i=1}^N (c,\gamma_i) & \tn{if } (\gamma_i)_i \in \Pi((\mu_i)_i) \\
				\infty & \tn{else.}
				\end{cases}		
	\end{align}
	The Wasserstein barycenter of the tuple $(\mu_i)_i$ can be obtained by optimizing $\mc{G}$, it is given by the common second marginals of the optimal $(\gamma_i)_i$.
\end{definition}

\begin{proposition}[Existence and Convergence of Barycenter]\label{prop-barycenters}
	Minimizers of the functionals $\mc{G}_k$ and $\mc{G}$ exist and $\mc{G}_k$ $\Gamma$-converges to $\mc{G}$ w.r.t.\ the narrow topology.
\end{proposition}
\begin{proof}
  The proof is very similar to the arguments given above and we only sketch the main steps. First note that we can limit our analysis to candidates in $\Pi((\mu_i)_i)$ with a common second marginal. Further, from the estimate \eqref{eq:EntropyMomentBound} we deduce that \correc{
\begin{align}
  \mc{G}_k((\gamma_i)_i)&=\sum_{i=1}^N\left(\int_{\R^n\times \R^n} \abs{x_i-y_i}^2\,d\gamma_i +\veps_k  H_2(\gamma_i)\right) \\
                        &\geq \sum_{i=1}^N\left(\int_{\R^n\times \R^n} \left(\frac{1}{2}\abs{y_i}^2-\abs{x_i}^2\right)\,d\gamma_i-\veps_k C\left(M_{\R^n\times \R^n}(\gamma_i)+1 \right)^\alpha\right)\\
                        &\geq\sum_{i=1}^N\left( \frac{1}{2}M_{\R^n}(\rho)-M_{\R^n}(\mu_i)-\veps_k C\left(M_{\R^n}(\rho)+M_{\R^n}(\mu_i)+1 \right)^\alpha\right),
\end{align}
where each term of the sum is coercive as a function of $M_{\R^n}(\rho)$. Hence it suffices to consider feasible sequences where the common second marginal $\rho$ has uniformly bounded second moment $M(\rho)$.
}

	The functionals $\mc{G}_k$ and $\mc{G}$ can then essentially be written as sums of the functionals $\mc{F}_k$ and $\mc{F}$. Existence of minimizers and the lim-inf condition then follow from standard narrow compactness and l.s.c.\ arguments as above. Likewise, a recovery sequence for the lim-sup condition can be constructed by applying the `block approximation' to each $\gamma_i$ separately.
\end{proof}


\section{Entropy Regularized Wasserstein Gradient Flows}
\label{sec:JKO}

\subsection{Set-up}
\label{sec:JKO-Setup}

In this section we investigate a time discrete gradient descent scheme on $\prob(\R^n)$, where we replace the standard optimal transport distance \eqref{eq:W2} by its entropy regularized variant \eqref{eq:W2Entropy}. This is no longer a distance, but we show that in a suitable joint limit of vanishing entropy regularization and time-step size it converges to the same PDE as the unregularized scheme. It is well-known since the seminal work of Jordan, Kinderlehrer and Otto \cite{JKO1998} that under appropriate assumptions on the energy $F$, the JKO implicit Euler scheme:
\begin{equation}\label{jkoschemeusual}
\rho^{k+1} \in \argmin_{\rho} \{\frac{1}{2\tau} W^2(\rho, \rho^k)+F(\rho)\}
\end{equation}
selects as the time step $\tau\to 0$ a solution of the evolution PDE:
\begin{equation}\label{pdejko}
\partial_t \rho=\ddiv(\rho \nabla F'(\rho)).
\end{equation}
We refer to the textbook \cite{AmbrosioGradientFlows2005} for a detailed presentation of the theory of gradient flows in the Wasserstein space and general convergence results. From a numerical point of view, the Wasserstein term in \eqref{jkoschemeusual} is delicate to handle and the regularized JKO scheme consists in replacing $W^2$ by $W^2_\veps$ in \eqref{jkoschemeusual}. Doing so, we have an Euler scheme which involves two small paramaters $\veps$ and $\tau$ and the question we wish to address is how to relate these parameters in such a way that the scheme still converges to a solution of the evolution equation \eqref{pdejko}.

The functional which we have to consider at each step of the regularized JKO scheme therefore is given by, for $\mu, \nu \in \probac(\R^n)$:
\begin{align}
	J_{\veps,\tau}(\mu,\nu) & \eqdef \frac{1}{2\,\tau} W_\veps^2(\mu,\nu) + F(\nu)
	\label{eq:IntroJKOEntropy}
\end{align}
where $\veps>0$ is the entropy regularization parameter and $\tau>0$ is the time-step size. $F : \probac(\R^n) \rightarrow \RExt$ is called \emph{free energy} and describes external potentials and local self-interactions. We now specify the class of free energies that we consider in this article, to keep the exposition simple, we did not try to give the most general assumptions.

\begin{assumption}[Free Energy]
\label{asp:FreeEnergy}
Assume that $F : \probac(\R^n) \rightarrow \RExt$ has the following form:
\begin{equation}
	F(\nu) \eqdef V(\nu) + U(\nu)
\end{equation}
\begin{align}
\tn{with} \quad 
	V(\nu) & \eqdef \int v(x)\,\nu(x)\,dx\,, & U(\nu) & \eqdef \int u(\nu(x))\,dx\,.
\end{align}
\correc{The functionals $U$ and $V$ are respectively the internal and potential energies.} We call $v$ the \emph{potential \correc{density}} and assume that it is \emph{non-negative} ($v : \R^n \rightarrow \R_+$) and \emph{Lipschitz} (hence Lebesgue a.e.~differentiable);  $u : [0,\infty) \rightarrow \R$ represents the \emph{internal energy \correc{density}}. We assume it to be \emph{convex}, \emph{twice differentiable} ($u \in C^2((0,\infty),\R)$), $u(0)=0$ and \emph{super-linear}:
\begin{align}
	\lim_{s \rightarrow \infty} u(s)/s = \infty.
	\label{eq:IntroUSuperLinear}
\end{align}
Moreover, we assume that there exists a finite, positive $C$ and some $\alpha$ with $\frac{n}{n+2} < \alpha < 1$ such that
\begin{align}\label{eq:IntroULBound}
	u(s) \geq -C\,s^\alpha\,.
\end{align}
We introduce the \emph{pressure} $p$ associated to $u$:
\begin{equation*}	
	p : [0,\infty) \rightarrow \R_+, \qquad p(s) \eqdef s \cdot u'(s)-u(s).
\end{equation*}
Note that this is a non-negative and non-decreasing function. We further assume that for some $C>0$ and $m\ge 1$:
\begin{equation}
	\label{growth1}
	p(s) \le C s^m, \qquad p'(s) \ge \frac{s^{m-1}}{C}
\end{equation}
and 
\begin{equation}
	\label{growth2}
	U(\mu)+C \cdot M(\mu)\ge \frac{1}{C} \int_{\R^n} \mu(x)^m\,dx, \qquad \forall \mu\in \probac(\R^n)\,,
\end{equation}
which implies that bounds on the internal energy and on second moments give a $L^m$ bound as well.
\end{assumption}

\begin{remark}
	\label{rem:FreeEnergyUImplications}
	The above assumptions on $u$ also imply that it is \emph{bounded from below}, it is \emph{continuous in $0$}, $\lim_{s \rightarrow 0^ +} s \cdot u'(s) = 0$ and \eqref{eq:IntroULBound} implies that there exists a (different) constant $C>0$ such that (cf.~\thref{thm:EntropyMomentBound})
	\begin{align}
		\label{eq:IntroFLBound}
		U(\mu) \geq -C\,(1 + M(\mu))^\alpha\,.
	\end{align}
\end{remark}

The PDE \eqref{pdejko} can be written as the nonlinear diffusion equation which is naturally formulated in terms of the pressure:
\begin{equation}\label{porouslike}
	\partial_t \rho=\Delta p(\rho)+\ddiv(\rho \nabla v)\,
\end{equation}
that we supplement with an initial condition $\rho\vert_{t=0}=\rho_0$ with $\rho_0\in \prob(\R^n)$ such that $F(\rho_0)<+\infty$ (so that $\rho_0\in L^m(\R^n)$). Note that the assumptions above for $u$ ($U$ respectively) are satisfied for the entropy and convex power functions which correspond repectively to the heat equation and the porous medium equation. The next paragraphs are devoted to the convergence proof of the following regularized JKO scheme.

Let $\veps,\tau > 0$ be real positive parameters and $N>0$ some positive integer.
For a given initial density $\iteratePar{\rho}{0}=\rho_0 \in \probac(\R^n)$ with $F(\iteratePar{\rho}{0}) < \infty$ we now introduce the sequence $\{\iteratePar{\rho}{k}\}_{k=0}^{N-1}$ in $\probac(\R^n)$ constructed recursively from $\iteratePar{\rho}{0}$ by:
\begin{equation}\label{eq-jko-step-regul}
	\iteratePar{\rho}{k+1} \in \argmin_{\rho \in \probac(\R^n)} J_{\veps,\tau}(\iteratePar{\rho}{k},\rho)
	\qquad \text{for} \qquad k=0,\ldots,N-2.
\end{equation}
We denote by $\iteratePar{\gamma}{k+1}$ the optimal coupling in $W_\veps^2(\iteratePar{\rho}{k},\iteratePar{\rho}{k+1})$.

For a given sequence we introduce the piecewise constant interpolation
\begin{equation}
	\interpEpsTauN : [0,N \cdot \tau ) \, \rightarrow \, \probac(\R^n), \qquad
	\interpEpsTauN(t) = \iteratePar{\rho}{k} \quad \tn{when} \quad t \in [k \cdot \tau, (k+1) \cdot \tau )\,.
\end{equation}

Our main result, stated in Theorem \ref{thm:JKOConvergence}, is the convergence of this scheme to a solution of the PDE \eqref{porouslike} provided $\veps \vert \log(\veps) \vert=O(\tau^2)$. We do not know if the condition $\veps \vert \log(\veps) \vert=O(\tau^2)$ is optimal but is not difficult to see that it is necessary that $\veps=o(\tau)$ (see in particular Proposition  \ref{thm:Eulerreg}).

%


\subsection{Preliminary Results}
We now establish some preliminary results that are fundamental for the further analysis of the gradient flow.
\begin{lemma}
	\thlabel{thm:Flsc}
	$F$ is l.s.c.~in the narrow topology under bounded second order moments.
\end{lemma}

\begin{proof}
	Narrow lower semi-continuity of $V$ follows from \cite[Thm.~2.38]{AFP00} ($(x,y) \mapsto v(x)\cdot y$ is l.s.c.\ in $(x,y)$ and 1-homogeneous and convex in $y$).
	For $U$ we apply \thref{cor:fctLSC} (subject to Assumption \ref{asp:FreeEnergy} and bounded moments).
\end{proof}


\begin{proposition}
	For $\mu \in \probac(\R^n)$, $H_1(\mu)$, $F(\mu) < \infty$, the optimization problem
	\begin{align}\label{eq-min-jko}
		\inf_{\nu \in \probac(\R^n)} J_{\veps,\tau}(\mu,\nu)
	\end{align}
	has a unique minimizer.
\end{proposition}

\begin{proof}
	\correc{Using the fact that the infimum of a sum is larger than the sum of infima and that, among transport plans, $H_2$ is bounded from below by the sum of $H_1$ entropies of the marginals, we have,}  for any $\nu \in \probac(\R^n)$
	\begin{equation}
		W^2_\veps(\mu,\nu) \geq W^2(\mu,\nu) + \veps\,H_1(\mu) + \veps\,H_1(\nu)\,.
	\end{equation}
	This is obtained by optimizing separately over the transport and the entropy term in \eqref{eq:W2Entropy}. According to \eqref{eq:W2MomentBound} have $M(\nu) \leq 2 M(\mu) + 2\,W^2(\mu,\nu)$. With \thref{thm:EntropyMomentBound} and Assumption \ref{asp:FreeEnergy} we then have (for fixed $\mu$)
	\begin{equation}
		J_{\veps,\tau}(\mu,\nu) \geq C_1\,M(\nu) - C_2\,(1+M(\nu))^\alpha
	\end{equation}
	for two finite $C_1,C_2>0$ and some $\alpha \in (0,1)$. This implies that $\nu \mapsto J_{\veps,\tau}(\mu,\nu)$ is bounded from below.
	Existence of some $\nu \in \probac(\R^n)$ with $J_{\veps,\tau}(\mu,\nu)<\infty$ can be shown via explicit construction, as for example done in the proof of \thref{thm:OneStepBound}.
	Let $(\nu_h)_{h\in\N}$ be a minimizing sequence. Clearly $M(\nu_h)$ and $H_1(\nu_h)$ are \correc{uniformly} bounded. Let then $C<\infty$ be a uniform upper bound on $M(\nu_h)$ and $H_1(\nu_h)$. As $M(\nu_h)$ is bounded, the set $\{\nu_h\}_{h \in \N}$ is tight and hence there is a subsequence narrowly converging to some $\nu \in \probAll(\R^n)$. It follows that $M(\nu)$, $H_1(\nu) \le C$ as well ($H_1$ is l.s.c.~under bounded moments, \correc{see Definition \ref{def:Entropy} and~\thref{cor:fctLSC}}) and thus $\nu \in \probac(\R^n)$. It follows from \thref{thm:W2Epslsc} and \thref{thm:Flsc} that $\nu$ is a minimizer. Uniqueness follows from the strict convexity of $W^2_\veps(\mu,\cdot)$ (\thref{thm:W2EpsStrictConvexity}), and the convexity of $F$.
\end{proof}

\subsection{Euler-Lagrange Equations}

\begin{proposition}
	\thlabel{thm:IntroFVariation}
	Let $\nu\in  \probac(\R^n) \cap L^m(\R^n)$,  $w \in C^\infty_c(\R^n,\R^n)$ and let $(\eta, x)\in [0, \infty)\times \R^n \mapsto \phi_{\eta}(x)$ be the flow of $w$ (i.e. $\phi_0=\id$ and $\partial_\eta \phi_\eta=w(\phi_\eta)$), then the first variation of $F$ along the flow of $w$ is given by
	\begin{align}
		\delta F(\nu,w) & \eqdef \left. \frac{d}{d\eta} F\left({\phi_\eta}_\sharp\,\nu \right) \right|_{\eta=0} \nonumber 	\\
		& = -\int_{\R^n} \underbrace{\left(u'(\nu(x))\,\nu(x) - u(\nu(x))\right)}_{= p(\nu(x))}\,\ddiv(w)(x)\,dx + \int \la \nabla v(x), w(x) \ra\,\nu(x)\,dx\,.
		\label{eq:IntroFVariation}
	\end{align}
\end{proposition}
\begin{proof}
The proof being standard, we just give the main arguments for the sake of completeness. Denoting by $J_\eta(x):=\det(D_x \phi_\eta(x))$ the Jacobian of the diffeomorphism $\phi_\eta$, we have  
\[U({\phi_\eta}_\sharp\,\nu)=\int_{\R^n} u\Big(\frac{\nu(x)}{J_\eta(x)} \Big) J_\eta(x) \mbox{d} x, \; \partial_\eta J_\eta(x)=\ddiv(w(\phi_\eta(x))) J_\eta(x),\]
the integrand above therefore is a smooth function of $\eta$: 
\[\partial_\eta \Big( u\Big(\frac{\nu(x)}{J_\eta(x)} \Big) J_\eta(x)  \Big)=-p\Big(\frac{\nu(x)}{J_\eta(x)} \Big) J_\eta(x) \ddiv(w(\phi_\eta(x))),\]
thus, using Assumption \eqref{growth1}, the fact that $\nu\in L^m$ and Lebesgue's dominated convergence Theorem, we get
\[\frac{d}{d\eta} U\left({\phi_\eta}_\sharp\,\nu \right) |_{\eta=0}= -\int_{\R^n} p(\nu(x))\ddiv(w)(x)\,dx.\]
The term involving the potential energy $V$ is similar and therefore omitted. 

\end{proof}

\label{sec:EulerLagrange}
\begin{proposition}[Euler-Lagrange Equation]
       \thlabel{thm:Eulerreg}
	Let $\nu$ be the optimizer of $J_{\veps,\tau}(\mu,\cdot)$, let $\gamma$ be the optimal transport plan in $W_\veps^2(\mu,\nu)$. Then for any $w \in C^\infty_c(\R^n,\R^n)$ one has
	\begin{equation}
		0 = - \frac{1}{\tau} \int \la w(y),x-y \ra\,d\gamma(x,y) - \frac{\veps}{2\,\tau} \int \nu(y)\,(\ddiv w)(y)\,dy + \delta F(\nu,w)
		\label{eq:EulerLagrangeEulerLagrange}
	\end{equation}
	where $\delta F(\nu,w)$ is given by \eqref{eq:IntroFVariation}.
\end{proposition}
\begin{proof}
The proof closely follows the arguments leading to \cite[Eq.\ (40)]{JKO1998}.
Let $\phi_\eta$ be flow of $w$ at time $\eta$, i.e.~$\phi_\eta$ satisfies the ODE $\partial_\eta \phi_\eta = w \circ \phi_\eta$ with initial condition $\phi_0 = \id$. Since $\nu$ is optimal, one has $J(\mu,{\phi_\eta}_\sharp \nu) \geq J(\mu,\nu)$ for all $\eta$.
Now study
\begin{align}	
	0 & \leq \limsup_{\eta \searrow 0} \frac{1}{\eta} \left(
		J_{\veps,\tau}(\mu,{\phi_\eta}_\sharp \nu) - J_{\veps,\tau}(\mu,\nu) \right) \\
	& \leq \limsup_{\eta \searrow 0} \frac{1}{2\,\tau\,\eta} \left(  W_\veps^2(\mu,{\phi_\eta}_\sharp \nu) - W_\veps^2(\mu,\nu) \right)
		+ \limsup_{\eta \searrow 0} \frac{1}{\eta} \left(  F({\phi_\eta}_\sharp \nu) - F(\nu) \right).
	\label{eq:EulerLagrangeLimSup}
\end{align}
The limit for the $F$ term is given by \eqref{eq:IntroFVariation}. Now let us look at the first limit. Given $\gamma$, we can generate a coupling ${(\id,\phi_\eta)}_\sharp \gamma \in \PiAc(\mu,{\phi_\eta}_\sharp \nu)$ to get an upper bound for $W_\veps^2(\mu,{\phi_\eta}_\sharp \nu)$. We find:
\begin{align}
	& \limsup_{\eta \searrow 0} \frac{1}{\eta} \left(
		W_\veps^2(\mu,{\phi_\eta}_\sharp \nu) - W_\veps^2(\mu,\nu) \right) \\
	\leq {} & \limsup_{\eta \searrow 0} \frac{1}{\eta} \left(
		\big(c,{(\id,\phi_\eta)}_\sharp \gamma\big) + \veps\,H_2\big( {(\id,\phi_\eta)}_\sharp \gamma \big) - (c,\gamma) - \veps\,H_2(\gamma) \right) \\
	= {} & \limsup_{\eta \searrow 0} \frac{1}{\eta} \left(
		\int \left[ |x-\phi_\eta(y)|^2 - |x-y|^2 \right]\,d\gamma(x,y) -
		\veps \int \gamma(x,y) \log \left| \det D\phi_\eta(y) \right| \,dx\,dy
		\right). \\
	\intertext{We now go to the limit:}
	\label{eq:EulerLagrangeWEntropyVariation}
	= {} & - 2 \int \la w(y),x-y \ra\,d\gamma(x,y) - \veps \int \nu(y)\,(\ddiv w)(y)\,dy.
\end{align}
	Plugging this into \eqref{eq:EulerLagrangeLimSup} we find:
	\begin{equation}
		0 \leq - \frac{1}{\tau} \int \la w(y),x-y \ra\,d\gamma(x,y) - \frac{\veps}{2\,\tau} \int \nu(y)\,(\ddiv w)(y)\,dy + \delta F(\nu,w).
	\end{equation}
	By using linearity of the r.h.s.~w.r.t.~$w$ and the fact that this equation holds simultaneously for $w$ and $-w$, we obtain equality.
\end{proof}
	



\subsection{A Priori Estimates}
\subsubsection{One Step}
Throughout this sub-section we will simply write $\iterate{\rho}{k}$ instead of $\iteratePar{\rho}{k}$, since no sequences for different parameters are compared. We denote by $C$ various positive, finite constants (constant w.r.t.~$\veps$, $\tau$, $k$, $\iteratePar{\rho}{k}$ and $\iteratePar{\gamma}{k}$).

\begin{proposition}
\thlabel{thm:OneStepBound}
For
\begin{align}
	0 < \veps \leq \veps |\log \veps| \leq C_1 \cdot \tau^2
	\label{eq:OneStepEpsBound}
\end{align}
with $C_1 \in (0,\infty)$, there is a positive constant $C_2 \in (0,\infty)$ such that
\begin{align}
	(c,\iterate{\gamma}{k+1}) & \leq \tau^2\,C_2 - \veps\,H_1(\iterate{\rho}{k+1}) + 2\,\tau\,F(\iterate{\rho}{k}) - 2\,\tau\,F(\iterate{\rho}{k+1})\,.
	\label{eq:OneStepBound} \\
	\intertext{and}
	(c,\iterate{\gamma}{k+1}) & \leq \tau^2\,C_2 + \tau^2\,C_1\,\Hneg(\iterate{\rho}{k+1}) + 2\,\tau\,F(\iterate{\rho}{k}) - 2\,\tau\,F(\iterate{\rho}{k+1})\,.
	\label{eq:OneStepBoundHneg}
\end{align}
\end{proposition}

\begin{proof}
Let $G \in C_c^\infty(\R^n,\R_+)$ be a compactly supported, smooth probability kernel, such that
\begin{align}
	\int G(x)\,dx & = 1\,, & \int x\,G(x)\,dx & = 0\,, & \int |x|^2\,G(x)\,dx & = 1\,, & H_1(G) & < \infty\,.
\end{align}
For a yet to be determined scale parameter $\sigma > 0$ denote by $G_\sigma$ the scaled kernel $G_\sigma(x) \eqdef \sigma^{-n} \, G(x/\sigma)$.
Let now
\begin{align}
	\hat{\gamma}(x,y) & \eqdef \iterate{\rho}{k}(x)\,G_\sigma(x-y)\,, &
	\hat{\rho}(y) & \eqdef \int \hat{\gamma}(x,y)\,dx\,.
\end{align}
Clearly $\hat{\gamma} \in \PiAc(\iterate{\rho}{k},\hat{\rho})$.
We find:
\begin{align}
	(c,\hat{\gamma}) & = \sigma^2 \\
	H_2(\hat{\gamma}) & = \int \hat{\gamma}(x,y)\,\log(\hat{\gamma}(x,y))\,dx\,dy
		= H_1(\iterate{\rho}{k}) + H_1(G) - n\,\log \sigma. \\
	\intertext{For $F(\hat{\rho})$ one finds (by Jensen's inequality used on $u$ and by the Lipschitz property of $v$):}
	\label{eq:OneStepFSmoothing}
	F(\hat{\rho}) & \leq U(\iterate{\rho}{k}) + V(\iterate{\rho}{k}) + C \cdot \sigma = F(\iterate{\rho}{k}) + C \cdot \sigma
\end{align}
where the constant $C$ contains the Lipschitz constant $L$ and the finite ``first moment'' $\int G(x)\,|x|\,dx$ of $G$. We fix now $\sigma = \sqrt{\veps}$ and summarize:
\begin{align}
	J_{\veps,\tau}(\iterate{\rho}{k},\hat{\rho}) & \leq \frac{1}{2\,\tau} \left(
		(c,\hat{\gamma}) + \veps\,H_2(\hat{\gamma}) \right)
		+ F(\hat{\rho}) \nonumber \\
	& \leq \frac{\veps}{\tau} \left( C + |\log\veps|\,C \right) + \frac{\veps}{2\,\tau} H_1(\iterate{\rho}{k}) + \sqrt{\veps} \, C + F(\iterate{\rho}{k}).
	\label{eq:OneStepRight}
\end{align}
For any $\gamma \in \PiAc(\iterate{\rho}{k},\iterate{\rho}{k+1})$ have
\begin{align}
	H_2(\gamma) & \geq H_1(\iterate{\rho}{k}) + H_1(\iterate{\rho}{k+1})\,. \\
	\intertext{Consequently}
	J_{\veps,\tau}(\iterate{\rho}{k},\iterate{\rho}{k+1}) & = \frac{1}{2\,\tau} \left(
		(c,\iterate{\gamma}{k+1}) + \veps\,H_2(\iterate{\gamma}{k+1}) \right) + F(\iterate{\rho}{k+1}) \nonumber \\
		& \geq \frac{1}{2\,\tau} \left(
		(c,\iterate{\gamma}{k+1}) + \veps\,H_1(\iterate{\rho}{k}) + \veps\,H_1(\iterate{\rho}{k+1}) \right) + F(\iterate{\rho}{k+1})\,.
		\label{eq:OneStepLeft}
\end{align}
Since we know $J_{\veps,\tau}(\iterate{\rho}{k},\iterate{\rho}{k+1}) \leq J_{\veps,\tau}(\iterate{\rho}{k},\hat{\rho})$ for any $\hat{\rho} \in \probac(\R^n)$ we combine \eqref{eq:OneStepRight} and \eqref{eq:OneStepLeft} to obtain
\begin{align}
	(c,\iterate{\gamma}{k+1}) & \leq
		\veps \left(1 + |\log \veps|\right)\,C
		- \veps\,H_1(\iterate{\rho}{k+1})
		+ \tau\,\sqrt{\veps}\,C + 2\,\tau\,F(\iterate{\rho}{k}) - 2\,\tau\,F(\iterate{\rho}{k+1}) \,. 
\end{align}
The statement follows now from assumption \eqref{eq:OneStepEpsBound}.
\end{proof}

\subsubsection{Multiple Steps}
\begin{proposition}[Bounding Second Moments.]
\thlabel{thm:MultiStepMoments}
Let $T>0$ be a real, positive, finite constant. Then there is a constant $0 < C < \infty$ such that for any pair of parameters $(\tau, N)$ with $\tau \cdot N \leq T$ and any $\veps$ satisfying \eqref{eq:OneStepEpsBound}, we have
\begin{align}
	\label{eq:MultiStepMomentsUniform}
	M(\iteratePar{\rho}{k}) < C
\end{align}
for any $k \leq N$.
\end{proposition}

\begin{proof}
Write $\iterate{\rho}{k}$ for $\iteratePar{\rho}{k}$ and $M_k$ for $M(\iterate{\rho}{k})$. Using \eqref{eq:W2MomentBound} we have
\begin{align*}
	M_k & \leq 2\,M_0 + 2\,W^2(\iterate{\rho}{0},\iterate{\rho}{k}) \\
	& \leq 2\,M_0 + 2\,\left( \sum_{l=0}^{k-1} \sqrt{W^2(\iterate{\rho}{l},\iterate{\rho}{l+1})} \right)^2 \\
	& \leq 2\,M_0 + 2\,k\,\left( \sum_{l=0}^{k-1} W^2(\iterate{\rho}{l},\iterate{\rho}{l+1}) \right)\,. \\
	\intertext{The second inequality is due to the triangle inequality for $W_2$, the third is due to the Cauchy-Schwarz inequality. Since $W^2(\iterate{\rho}{l},\iterate{\rho}{l+1}) \leq (c,\iterate{\gamma}{l+1})$ and with \eqref{eq:OneStepBoundHneg} we get:}
M_k	& \leq 2\,M_0 + 2\,k\,\left(
		k\,\tau^2\,C
		+ \tau^2\,C\sum_{l=1}^{k} \Hneg(\iterate{\rho}{l})
		+ 2\,\tau F(\iterate{\rho}{0}) - 2\,\tau\,F(\iterate{\rho}{k})
		\right). \\
	\intertext{Since $k \leq N$ and $N \cdot \tau \leq T$, with \thref{thm:EntropyMomentBound} and Assumption \ref{asp:FreeEnergy} we obtain for some $\alpha < 1$:}
	M_k & \leq C + \tau \, C \sum_{l=1}^{k} (M_l + 1)^\alpha + C\,(M_k + 1)^\alpha. \\
	\intertext{Let now $\ol{M}_k = \max_{0 \leq l \leq k} M_l$. Then}
	M_k & \leq C \left(1+ (\ol{M}_k + 1)^\alpha \right)
\end{align*}
and eventually $\ol{M}_k \leq C \left(1+ (\ol{M}_k + 1)^\alpha \right)$ which implies $\ol{M}_k \leq C$ for a suitable constant $C$.
\end{proof}

Analogous to \thref{thm:MultiStepMoments} one has the following bounds.

\begin{corollary}
	There exists a $C \in (0,\infty)$ such that for every $k \leq N$
	\begin{align}
		\label{eq:MultiStepEntropyUniform}
		H_1(\iteratePar{\rho}{k}) & \geq -C\,, &
		\Hneg(\iteratePar{\rho}{k}) & \leq C\,, &
		F(\iteratePar{\rho}{k}) & \geq -C\,.
	\end{align}
\end{corollary}
\begin{proof}
This follows immediately from \thref{thm:MultiStepMoments}, \thref{thm:EntropyMomentBound} and Assumption \ref{asp:FreeEnergy} (see also Remark \ref{rem:FreeEnergyUImplications}).
\end{proof}
\begin{corollary}[All Steps]
	\thlabel{thm:GammaAllSteps}
	There is a finite, positive $C$ such that
	\begin{equation}
		\sum_{k=0}^{N-1} (c,\iteratePar{\gamma}{k+1}) \leq \tau\,C\,.
	\end{equation}
\end{corollary}
\begin{proof}
	Sum up inequality \eqref{eq:OneStepBoundHneg} and use \eqref{eq:MultiStepEntropyUniform}.
\end{proof}

\begin{corollary}\thlabel{thm:Holderest}
For $i,j$ with $0 \leq i \leq j \leq N-1$, we have
\begin{align}
	W^2(\iteratePar{\rho}{i},\iteratePar{\rho}{j}) & \leq \left( \sum_{k=i}^{j-1} W(\iteratePar{\rho}{k},\iteratePar{\rho}{k+1}) \right)^2
		\leq (j-i) \cdot \sum_{k=0}^{N-1} W^2(\iteratePar{\rho}{k},\iteratePar{\rho}{k+1}) \nonumber \\
	& \leq (j-i) \cdot \sum_{k=0}^{N-1} (c,\iteratePar{\gamma}{k}) \leq (j-i) \cdot \tau \cdot C.
\end{align}
\end{corollary}

\begin{corollary}
	\thlabel{thm:FUBoundUniform}
	There exists a $C<\infty$ such that for all $k \leq N$, one has $F(\iteratePar{\rho}{k}) < C$.
\end{corollary}
\begin{proof}
	Since $(c,\iteratePar{\gamma}{k})\geq 0$ get from summing \eqref{eq:OneStepBoundHneg} over $l=0,\ldots,k-1$:
	\begin{equation*}
		F(\iteratePar{\rho}{k}) \leq k\,\tau\,C + \tau\,C \sum_{l=0}^{k-1} \Hneg(\iteratePar{\rho}{l+1}) + F(\iteratePar{\rho}{0}), 
	\end{equation*}
	with uniform bound on $\Hneg$, \eqref{eq:MultiStepEntropyUniform} and $k \leq N$, $N \cdot \tau \leq T$ get $F(\iteratePar{\rho}{k}) \leq C$.
\end{proof}

\subsection{Convergence}

The a priori estimates from the previous paragraph being roughly the same as for the standard JKO scheme, getting enough compactness to be able to pass to the limit in the Euler-Lagrange equation can be done by rather standard arguments. 

\begin{proposition}
	\thlabel{thm:TVBound}
	Let $T \in (0,\infty)$ and $N \in \N$, $\tau > 0$ such that $N \cdot \tau = T$.
	The function $(\interpEpsTauN)^m$ is in $L^1([0,T], W^{1,1}(\R^n))$, and there is a constant $C \in (0,\infty)$ (which does not depend on $\veps$, $\tau$ nor $N$), such that 
\begin{equation}
	\int_0^T \int_{\R^n} \left((\interpEpsTauN)^m(t,y) + \vert \nabla (\interpEpsTauN)^m\vert(t,y)\right)\,d y\,d t \le C,
\end{equation}
where the exponent $m$ is given by Assumption \ref{asp:FreeEnergy}, equations (\ref{growth1},\ref{growth2}).
\end{proposition}
\begin{proof}
Let us denote by $\iterate{\rho}{k}$ the density $\iteratePar{\rho}{k}$ and introduce the sequence of non-negative functions
	\begin{equation*}
		\iterate{\mu}{k} \eqdef \frac{\veps}{2\,\tau}\,\iterate{\rho}{k}+ p(\iterate{\rho}{k})\,.
	\end{equation*}
First, we study $\iterate{\mu}{k}$. From~\eqref{growth1}, \eqref{growth2}, and the fact that $\iterate{\rho}{k}\in L^1(\R^n)$ with finite energy $U$ and second moment (\thref{thm:MultiStepMoments,thm:FUBoundUniform}),  we see that $\iterate{\mu}{k}\in L^1(\R^n)$. Let us prove that in fact $\iterate{\mu}{k}\in W^{1,1}(\R^n)$.
	We start with the Euler-Lagrange equation~\eqref{eq:EulerLagrangeEulerLagrange} for step $k \geq 1$,
	\begin{equation*}
\begin{split}
0&=  -\tau^{-1} \int_{\R^n\times\R^n} \la w(y),x-y \ra \,d\iterate{\gamma}{k}(x,y)\\
&\qquad\qquad  - \int_{\R^n} \left(
			\left[\frac{\veps}{2\,\tau} \iterate{\rho}{k}(y) + p(\iterate{\rho}{k}(y)) \right]
				\ddiv w(y)
			+ \la \nabla v(y),w(y) \ra\,\iterate{\rho}{k}(y)
			\right) dy
\end{split}
	\end{equation*}
for all $w\in C^\infty_c(\R^n,\R^n)$. That is equivalent to 
\begin{equation}\label{eq:distribderiv}
  - \int_{\R^n} \iterate{\mu}{k}(x)\ddiv w(x)\,d x  = \tau^{-1} \int_{\R^n\times\R^n} \la w(y),x-y \ra \,d\iterate{\gamma}{k}(x,y)+ \int_{\R^n} \la \iterate{\rho}{k}(y)\nabla v(y),w(y) \ra dx.
\end{equation}
Since $\iterate{\gamma}{k}\in \PiAc(\iterate{\rho}{k-1},\iterate{\rho}{k})$, the disintegration theorem~\cite[Theorem 2.28]{AFP00} yields the existence of some measurable measure-valued map $y\mapsto \gamma_y^{(k)}$ such that $\iterate{\gamma}{k}= \gamma_y^{(k)}\otimes \iterate{\rho}{k}$ and $\abs{\gamma_y^{(k)}}(\R^n)=1$ for a.e.~$y\in\R^n$. Hence the first term on the r.h.s.\ may be written as 
\begin{equation}
  \tau^{-1} \int_{\R^n} \la w(y),\int_{\R^n} (x-y)\gamma_y^{(k)}(x)\,dx\ra \iterate{\rho}{k}(y)\,dy,
\end{equation}
and the function $y\mapsto \left(\int_{\R^n} (x-y)\gamma_y^{(k)}(x)\,dx\right) \iterate{\rho}{k}(y)$ is in $L^1(\R^n)$ since 
\begin{equation}\label{eq:majgamma}
  \int_{\R^n} \abs{\left(\int_{\R^n} (x-y)\gamma_y^{(k)}(x)\,dx\right)\iterate{\rho}{k}(y)}\,dy\leq \int_{\R^n\times \R^n} |x-y|\,d\iterate{\gamma}{k}(x,y) \leq (c,\iterate{\gamma}{k})^{1/2}.
\end{equation}
Moreover,  we see that $\iterate{\rho}{k}\nabla v\in L^1(\R^n)$, since
\begin{equation}\label{eq:majlip}
  \int_{\R^n}\abs{\iterate{\rho}{k}(x)\nabla v(x)}dx \leq \norm{\nabla v}_{\infty}<+\infty.
\end{equation}
As a result, the distributional derivative of $\iterate{\mu}{k}$ (see~\eqref{eq:distribderiv}) is representable by integration, and $\iterate{\mu}{k}\in W^{1,1}(\R^n)$. Combining~\eqref{eq:majgamma} and~\eqref{eq:majlip}, we also get 
\begin{equation}
  \norm{\nabla \iterate{\mu}{k}}_1\leq \tau^{-1}(c,\iterate{\gamma}{k})^{1/2}+\norm{\nabla v}_{\infty}.
\end{equation}
Summing over $k\in\{0,\ldots,N-1\}$, using the Cauchy-Schwarz inequality, \thref{thm:GammaAllSteps} and the fact that $N \cdot \tau \leq T < \infty$ we obtain
\begin{equation}\label{eq:majglob}
  \tau \sum_{k=0}^{1} \norm{\nabla \iterate{\mu}{k}}_1  \leq  \sqrt{N} \left(\sum_{k=0}^{N-1} (c,\iterate{\gamma}{k}) \right)^{1/2} + T\norm{\nabla v}_{\infty}
			\leq C,\end{equation}
    for some constant $C \in (0,\infty)$ which is independent of $N$, $\veps$, $\tau$. 

\newcommand{\fmap}{f_{\veps,\tau}}
Now, we turn to $\iterate{\rho}{k}$. Let $\fmap:[0,+\infty)\rightarrow [0,+\infty)$ be defined by $\fmap(s)\eqdef\frac{\veps}{2\,\tau}s + p(s)$. The mapping $\fmap$ is a homeomorphism with $\fmap(0)=0$, and one may check that its inverse $\fmap^{(-1)}$ is $\frac{2\tau}{\veps}$-Lipschitzian (using the non-negativity and monotonicity of $p$). By~\cite[Theorem 3.99]{AFP00}, we see that $\iterate{\rho}{k}=\fmap^{(-1)}(\iterate{\mu}{k})$ is in $W^{1,1}(\R^n)$, and $\nabla \iterate{\rho}{k}=(\fmap^{(-1)})'(\iterate{\mu}{k}) \nabla\iterate{\mu}{k}$. Since for $t>0$, by~\eqref{growth1},
\begin{equation*}
  (\fmap^{(-1)})'(t)=\left(\fmap'\left(\fmap^{(-1)}(t)\right)\right)^{-1}= \left(\frac{\veps}{2\,\tau} + p'\left(\fmap^{(-1)}(t)\right)\right)^{-1}\leq  \left(\frac{\veps}{2\,\tau} +\frac{1}{C} \left(\fmap^{(-1)}(t)\right)^{m-1}\right)^{-1},
\end{equation*}
we obtain that $\abs{\nabla \iterate{\rho}{k}}\leq \left(\frac{\veps}{2\,\tau} +\frac{1}{C} \left(\iterate{\rho}{k}\right)^{m-1}\right)^{-1}\abs{\nabla\iterate{\mu}{k}}$, so that $\abs{(\iterate{\rho}{k})^{m-1}\nabla \iterate{\rho}{k}}\leq C\abs{\nabla\iterate{\mu}{k}}$. From that inequality, we deduce that $(\iterate{\rho}{k})^{m}\in W^{1,1}(\R^n)$, with
\begin{align*}
\abs{\nabla(\iterate{\rho}{k})^{m}}&\leq C\abs{\nabla\iterate{\mu}{k}}.
\end{align*}
From~\eqref{eq:majglob}, we deduce that 
\begin{equation}
	\int_0^T \int_{\R^n} \vert \nabla (\interpEpsTauN)^m\vert \le C.
\end{equation}
It remains to prove the estimate on $(\interpEpsTauN)^m$. By~\thref{thm:FUBoundUniform} and~\thref{thm:MultiStepMoments} and~\eqref{growth2} and summing over $k\in\{1,\ldots, N\}$, we see that 
\begin{equation*}
	\int_0^T \int_{\R^n} (\interpEpsTauN)^m \le C. \qedhere
\end{equation*}
  \end{proof}

\begin{proposition}
	As $\veps, \tau \rightarrow 0$ in a way respecting \eqref{eq:OneStepEpsBound}, after choosing a suitable subsequence, $\interpEpsTauN$ converges strongly in $L^m((0,T)\times \R^n)$ to some $\interp : [0,T] \mapsto \probac(\R^n)$. The associated pressures converge strongly in $L^1((0,T)\times \R^n)$
\end{proposition}
\begin{proof}
	Strong convergence heavily relies on a generalization of the Aubin Lions Lemma due to  Savar\'e and Rossi (see Theorem 2 in \cite{Rossi03}), the BV estimate above and then arguments similar to those developed by Di Francesco and Matthes \cite{Matthes14} and Laborde \cite{Laborde15}. For a fixed constant $C \in (0,\infty)$ let us denote by
	\begin{align}
		\mc{U} \eqdef \left\{\interpEpsTauN : \veps>0, \tau>0, N \in \N \tn{ s.t. } (\veps,\tau) \tn{ satisfy \eqref{eq:OneStepEpsBound} with constant } C, N \cdot \tau = T\right\}
	\end{align}
	a family of interpolated flows. First, as proved in \cite{Matthes14}, \thref{thm:Holderest}  implies that we have 
	\begin{equation}\label{compactentemps}
	\lim_{h \to 0} \sup_{\rho \in \mc{U}} \int_0^{T-h} W_2(\rho(t+h, .), \rho(t,.))dt=0.
	\end{equation}
Now the estimates of the previous sections also give
\begin{equation}\label{compactenx}
\sup_{\rho\in \mc{U}}  \int_0^T G(\rho(t,.))\, dt\le C<+\infty
\end{equation}	
where
\begin{align}
	G(\rho) \eqdef
	\begin{cases} 
		M(\rho)+ \int \rho^m\,dx + \int \vert D \rho^m\vert \mbox{ if } \rho\in \probac(\R^n), \\ 
		+\infty \mbox{ otherwise}
	\end{cases}
\end{align}
and $\int \vert D \rho^m\vert $ refers to the total variation of $\rho^m$, defined by 
\begin{equation}
  \int \vert D \rho^m\vert \eqdef \sup \left\{\int_{\R^n}\rho^m \ddiv \varphi\,dx : \varphi\in  C_c^\infty(\R^n,\R^n), \norm{\varphi}_\infty\leq 1 \right\}\in [0,+\infty].
\end{equation}
In the case where $\rho^m\in W^{1,1}(\R^n)$ (for instance, by \thref{thm:TVBound}, if $\rho = \interpEpsTauN(t,\cdot)$ for some $t\in[0,T)$), then  $\int \vert D \rho^m\vert=  \int \vert \nabla\rho^m(x)\vert\,dx$. We refer the reader to~\cite{AFP00} for more details about functions with bounded variation.

Thanks to \thref{thm:equic} below, sub-level sets of $G$ are relatively compact in $L^m (\R^n)$ so the Savar\'e-Rossi theorem gives the desired $L^m$ compactness. 
		
	The statement on the convergence of the pressure then follows from the assumption that $p(s)\le C \cdot s^m$, \eqref{growth1}, which then implies that $\mu\mapsto p(\mu)$ is continuous from $L^m$ to $L^1$.
	%

\end{proof}

\begin{lemma}\thlabel{thm:equic}
$G$ is lower semi-continuous in $L^m$ and its sub-level sets are compact.
\end{lemma}
\begin{proof}
The  proof is similar to arguments from \cite{Matthes14} and \cite{Laborde15}, we provide it here for the sake of completeness. The fact that $G$ is l.s.c.\ for the $L^m$ topology is standard.  

Now, let us prove the relative compactness of the sub-level sets. For some $C \in (0,\infty)$ let
\begin{align}
  A \eqdef \left\{ \eta \in \BV(\R^n) \cap L^{1/m}(\R^n) : \eta \geq 0,
		\int_{\R^n} (|x|^2\,\eta^{1/m} + \eta + |D \eta|)\, dx \leq C \right\}.
\end{align}
Since the map $\eta \mapsto \rho \eqdef \eta^{1/m}$ is continuous from $L^1$ to $L^m$ it is sufficient to show that $A$ is relatively compact in $L^1$.
The set $A$ being bounded in $\BV(\R^n)$, it is compact in $L^1_{loc}(\R^n)$~\cite[Th.~3.23]{AFP00}, that is, every sequence in $A$ has a subsequence which converges for all compact set $K\subset \R^n$ in $L^1(K)$. To prove relative compactness in $L^1(\R^n)$ (and not only $L^1_{loc}(\R^n)$), we have to control the mass on complements of balls, i.e.\ we have to prove that
\begin{equation*}
	\lim_{R\to \infty} \sup_{\eta\in A} \int_{\R^n : |x| >R} \eta\,dx =0\,.
\end{equation*}
If $n=1$, $\norm{\eta}_\infty \leq C(\int \vert \eta\vert + \int \vert D\eta\vert)$, so that we may write
\begin{equation}
  \int_{\R^n : |x| >R} \eta\, dx \leq \frac{1}{R^2}\int |x|^2 \eta^{1/m}\, dx \norm{\eta}_\infty^{1-1/m}\leq \frac{C}{R^2}.
\end{equation}
If $m>1$,  let us define
\begin{align*}
	p \eqdef n-\frac{n-1}{m}\geq 1, \qquad
	\alpha \eqdef \frac{1}{m\,p}, \qquad
	\beta \eqdef \frac{2}{p}\,.
\end{align*}
By H\"older's and a Sobolev inequality, we then have for all $\eta\in A$:
\begin{align*}
	& \int_{\R^n : |x| >R} \eta\, dx
		\leq \frac{1}{R^{\beta}} \int_{\R^n} |x|^{\beta} \, \eta^{\alpha} \, \eta^{1-\alpha} \, dx \\
	\leq {} & \frac{1}{R^{\beta}} \left(
		\int_{\R^n} |x|^{\beta\,p}\,\eta^{\alpha\,p}\,dx
		\right)^{1/p}
		\left( \int_{\R^n} \eta^{(1-\alpha)\, p/(p-1)}\,dx
			\right)^{(p-1)/p} \\
	= {} & \frac{1}{R^{\beta}} \left( \int_{\R^n} |x|^2\, \eta^{1/m}\,dx \right)^{1/p}
		\left( \int_{\R^n} \eta^{n/(n-1)}\,dx \right)^{(p-1)/p}
		\leq \frac{C}{R^{\beta}}
\end{align*}
which gives the desired result. 
As a result, the sub-level sets of $G$ are relatively compact, and by the lower semi-continuity property they are also closed, hence they are compact.
\end{proof}

\subsection{Convergence to the Nonlinear Diffusion PDE}
\label{sec:JKOConvergencePDE}
Now we are ready to prove our convergence result. Let $(\veps_k)_k$, $(\tau_k)_k$ be two sequences of positive regularization and time-step parameters such that 
\begin{equation}\label{convergencecontrolee}
\lim_{k \rightarrow \infty} \veps_k = \lim_{k \rightarrow \infty} \tau_k=0, \qquad
 -\veps_k \log(\veps_k)\le C \tau_k^2
\end{equation}
for a fixed constant $C \in (0,\infty)$ (c.f.~\eqref{eq:OneStepEpsBound}).
We also fix a time $T \in (0,\infty)$ and set the number of time-steps $N_k$ in such a way that 
\begin{equation}
\lim_{k \rightarrow \infty} N_k \tau_k =T.
\end{equation}
For notational convenience let us then set $\interp_k \eqdef \interp^{(\veps_k, \tau_k,N_k)}$ and let $\interp \in L^m((0,T)\times \R^n)$ be such that $\interp_k$ converges strongly in $L^m$ to $\interp$ and $p(\interp_k)$ converges to $p(\interp)$ strongly in $L^1$.
By virtue of \thref{thm:Holderest} and the refined version of Arzela-Ascoli Theorem in \cite[Proposition 3.3.1]{AmbrosioGradientFlows2005}, we may also assume that $\interp_k(t,.)$ converges narrowly to $\interp(t,.)$ uniformly in $t\in [0,T]$.


\begin{theorem}
\thlabel{thm:JKOConvergence}
In the above setup (Assumption \ref{asp:FreeEnergy} and \eqref{convergencecontrolee}),  the limit curve of measures $\interp$ solves the evolution equation
\begin{equation}\label{evolnonlinear}
\partial_t \interp=\Delta p(\interp)+\ddiv(\interp \nabla v), \; \interp|_{t=0}=\rho_0.
\end{equation} 
\end{theorem}

\begin{proof}
Let $\phi \in C_c^{\infty}([0,T)\times \R^n)$ and $k$ be large enough so that $\phi(t,.)=0$ for $t\in[(N_k-1)\tau_k, N_k \tau_k]$. We then have 
\begin{align*}
\int_0^T \int_{\R^n} (\partial_t \phi)\, d\interp_k &
	= \sum_{l=0}^{N_k-1} \int_{\R^n} [
		\phi((l+1)\,\tau_k,\cdot)-\phi(l\,\tau_k,\cdot)
		]\, d\rho^{(\veps_k, \tau_k, l)} \\
	& = -\int_{\R^n} \phi(0,\cdot) \rho_0
		+ \sum_{l=1}^{N_k-1} \int_{\R^n} \phi(l \tau_k, \cdot)\,
			(\rho^{(\veps_k, \tau_k, l-1)}-\rho^{(\veps_k, \tau_k, l)})\,dx\,.
\end{align*}
Using the optimal plan $\gamma^{(\veps_k, \tau_k, l)}$ for $W^2_{\veps_k}$ between $\rho^{(\veps_k, \tau_k, l-1)}$ and $\rho^{(\veps_k, \tau_k, l)}$ we can rewrite:
\begin{equation*}
	\int_{\R^n} \phi(l \tau_k, \cdot) (\rho^{(\veps_k, \tau_k, l-1)}-\rho^{(\veps_k, \tau_k, l)})\,dx = 
	\int_{\R^n\times \R^n} [\phi(l\, \tau_k, x)-  \phi(l\,\tau_k, y)] \, d\gamma^{(\veps_k, \tau_k, l)}(x,y)\,.
\end{equation*}
Taylor expanding the integrand as
\begin{equation*}
	\phi(l\, \tau_k, x)-  \phi(l\, \tau_k, y)
		= \nabla \phi(l\, \tau_k, y)\cdot (x-y)+R_{k,l}(\phi, x,y) \tn{ with } |R_{k,l}(\phi, x,y)| \le \frac{1}{2} \Vert \nabla^2\phi\Vert_{\infty} |x-y|^2
\end{equation*}
and using \thref{thm:GammaAllSteps}, we get that
\begin{align*}
	R_k(\phi) & \eqdef \sum_{l=1}^{N_k-1}
		\int_{\R^n\times \R^n} R_{k,l}(\phi, x,y)\,d\gamma^{(\veps_k, \tau_k, l)}(x,y) \\
	\intertext{satisfies}
 	|R_k(\phi)| & \leq C\, \tau_k \| D^2 \phi \|_{\infty,\infty} \to 0 \tn{ as } k \to \infty.
\end{align*}
We then have
\begin{equation}\label{cvgpde1}
	\int_0^T \int_{\R^n} (\partial_t \phi)\, d\interp_k =
		-\int_{\R^n} \phi(0,\cdot)\,d\rho_0 + \sum_{l=1}^{N_k-1} \int_{\R^n\times \R^n}
			\la \nabla  \phi(l \tau_k, y),x-y\ra \, d\gamma^{(\veps_k, \tau_k, l)}(x,y)+R_k(\phi).
\end{equation}
We then use the Euler-Lagrange equation as derived in \eqref{eq:EulerLagrangeEulerLagrange} with $\nabla \phi(l \tau_k,\cdot)$ as test-function, to rewrite
\begin{equation}\label{cvgpde2}
	\int_0^T \int_{\R^n} \partial_t \phi \interp_k=-\int_{\R^n} \phi(0,\cdot) \rho_0
		+ A_k + B_k + C_k + R_k(\phi)\,,
\end{equation}
where
\begin{equation}\label{cvgpde3}
	A_k \eqdef -\frac{\veps_k}{2} \sum_{l=1}^{N_k-1} \int_{\R^n}
		\Delta \phi(l\, \tau_k,\cdot)\,d\rho^{(\veps_k, \tau_k, l)},\quad
	B_k \eqdef -\tau_k \sum_{l=1}^{N_k-1} \int_{\R^n}
		\Delta \phi(l\, \tau_k,\cdot)\, p(\rho^{(\veps_k, \tau_k, l)})\,dx
\end{equation}
and 
\begin{equation}\label{cvgpde4}
	C_k \eqdef \tau_k \sum_{l=1}^{N_k-1} \int_{\R^n}
		\la \nabla \phi(l\, \tau_k,\cdot), \nabla v \ra\, d\rho^{(\veps_k, \tau_k, l)}\,.
\end{equation}
Thanks to the smoothness of $\phi$ and the fact that $\nabla v$ is bounded we have
\begin{align}
	B_k & = -\int_{\tau_k}^T \int_{\R^n} \Delta \phi\,p(\interp_k)\,dx
		+ \mc{O}(\tau_k \,\|\partial_t \Delta \phi \|_{\infty,\infty})\,, \\
	C_k & = \int_{\tau_k}^T \int_{\R^n}  \la \nabla  \phi, \nabla v \ra  d\interp_k
		+ \mc{O}(\tau_k \, \|\partial_t \nabla \phi \|_{\infty,\infty})\,.
\end{align}
As for $A_k$, it converges to $0$ since
\begin{equation}
	\tau_k  \sum_{l=1}^{N_k-1} \int_{\R^n} \Delta \phi(l \tau_k,\cdot)\, d\rho^{(\veps_k, \tau_k, l)} =
		\int_0^T \int_{\R^n} \Delta \phi\, \interp_k
			+ \mc{O}(\tau_k \|\partial_t \Delta \phi \|_{\infty,\infty})
\end{equation}
and $\veps_k / \tau_k$ converges to $0$ as $k \to \infty$.  
Putting everything together, we thus get 
\begin{equation}\label{cvgpde6}
	\int_0^T \int_{\R^n} (\partial_t \phi) \interp_k = 
		- \int_{\R^n} \phi(0,\cdot) d\rho_0 - \int_0^T \int_{\R^n} p(\interp_k)\, \Delta \phi\,dx
		+ \int_0^T \int_{\R^n}  \la \nabla  \phi, \nabla v \ra d\interp_k +S_k(\phi),
\end{equation} 
 where $S_k(\phi)\to 0$ as $k\to \infty$. Letting $k\to \infty$ and using the strong convergence of $\interp_k$ to $\interp$ in $L^m$ and the strong convergence of $p(\interp_k)$ to $p(\interp)$ in $L^1$, we thus get
\begin{equation}\label{cvgpde7}
	\int_0^T \int_{\R^n} (\partial_t \phi)\,d\interp =
		- \int_{\R^n} \phi(0,\cdot) d\rho_0 - \int_0^T \int_{\R^n} p(\interp)\, \Delta \phi\,dx
		+ \int_0^T \int_{\R^n} \la \nabla  \phi, \nabla v \ra d\interp,
\end{equation}
which is precisely the weak form of \eqref{evolnonlinear}.
\end{proof}

\correc{We have proven that, under the condition \eqref{convergencecontrolee}, every cluster point obtained by the regularized JKO-scheme is a solution of \eqref{evolnonlinear}. If there is uniqueness for \eqref{evolnonlinear} (which is the case  for instance if the energy $F$ is $\lambda$-displacement convex for some $\lambda\in \R$, see \cite{AmbrosioGradientFlows2005}), then by our compactness estimates, there is whole convergence of $\interp_k$ (not only up to a subsequence).}


\section{Numerical Illustrations}

\newcommand{\DiscFont}[1]{\mathbf{#1}}
\renewcommand{\DiscFont}[1]{{\rm #1}}

\newcommand{\bmu}{\DiscFont{m}} 
\newcommand{\brho}{\DiscFont{r}} 
\newcommand{\bgamma}{\DiscFont{G}} 
\newcommand{\bK}{\DiscFont{K}} 
\newcommand{\bA}{\DiscFont{a}} 
\newcommand{\bB}{\DiscFont{b}} 
\newcommand{\bU}{\DiscFont{u}} 
\newcommand{\bV}{\DiscFont{v}} 
\newcommand{\Ndisc}{p} 
\newcommand{\LambertW}{\tn{LambertW}}
\newcommand{\rhoSol}{\rho_{\tn{sol}}}
\newcommand{\bleb}{\DiscFont{L}}
\newcommand{\blebvol}{l}

This section showcases some 1-D and 2-D numerical simulations that illustrate the behavior of the entropic regularization of gradient flows, and in particular the limit of small $\veps$. For simulations related to the entropic regularization~\eqref{eq:W2} of the optimal transport problem, as studied in Section~\ref{sec:OT}, we refer to~\cite{BregmanProj2015}. 
We recall here the numerical scheme initially proposed in~\cite{Peyre-JKO} and later refined in~\cite{2016-chizat-sinkhorn}.

\subsection{Discretization}
We thus aim at numerically computing an approximation of each regularized JKO step~\eqref{eq-jko-step-regul}. 
For that purpose, we suppose that we are given a discrete grid $(x_i)_{i=1}^\Ndisc \subset \R^n$ to approximate the underlying continuous domain. We restrict our attention to discrete measures supported on this grid, which have the form 
\begin{equation*}
	\rho = \sum_{i=1}^\Ndisc \brho_i \delta_{x_i}
	\in \prob(\R^n), 
\end{equation*}
where $\brho = (\brho_i)_{i=1}^\Ndisc$ is a vector in the probability simplex $\Sigma_\Ndisc$
\eq{
	\Sigma_\Ndisc \eqdef \enscond{\brho \in \R_+^\Ndisc}{\sum_{i=1}^\Ndisc \brho_i=1}, 
} 
and $\delta_x$ is the Dirac measure located at $x \in \R^n$.
We assume that every grid point represents a volume element of uniform size $\blebvol$, such that the discrete approximation of the Lebesgue measure is given by $\bleb=\blebvol \sum_{i=1}^{\Ndisc} \delta_{x_i}$. The discrete approximation of the Lebesgue density of $\rho$ is then given by $(\brho_i/\blebvol)_i$.

As defined for instance in~\cite{BregmanProj2015}, the regularized optimal transport cost between two vectors $\bmu,\brho \in \Sigma_\Ndisc$ in the simplex reads
\eq{
	\ol{W}_\veps^2(\bmu,\brho) \eqdef \min_{\bgamma \in \R_+^{\Ndisc \times \Ndisc}} \enscond{
		\sum_{i,j=1}^{\Ndisc} c_{i,j} \bgamma_{i,j}
			+ \veps \, \bgamma_{i,j} \, \log(\bgamma_{i,j}/\blebvol^2)
		}{
			\bgamma \ones_\Ndisc = \bmu, 
			\bgamma^\top \ones_\Ndisc = \brho
		}, 
}
where $c_{i,j} \eqdef \norm{x_i-x_j}^2$.
Here we have used the shorthand notation $\bgamma \ones_\Ndisc = (\sum_j \bgamma_{i,j})_{i}$, where $\ones_\Ndisc=(1,\ldots,1)^\top \in \R^\Ndisc$, and $\bgamma^\top$ denotes the transpose of $\bgamma$. 
This corresponds to a discrete approximation of the definition~\eqref{eq:W2Entropy} of $W_\veps^2$, where the entropy $H_2$ with respect to the Lebesgue measure is replaced by the entropy relative to the discretized Lebesgue measure on the product grid (using the usual convention $0\log(0)=0$).
%

The infinite dimensional optimization~\eqref{eq-jko-step-regul} is then replaced by the following finite dimensional convex program
\begin{equation}\label{eq-jko-step-regul-discr}
	\iteratePar{\brho}{k+1} = \argmin_{\brho \in \Sigma_\Ndisc} \ol{J}_{\veps,\tau}(\iteratePar{\brho}{k},\brho)
	\qquad \text{for} \qquad k=0,\ldots,N-2, 
\end{equation}
where, given an arbitrary fixed probability vector $\bmu \in \Sigma_\Ndisc$, the discretized functional reads
\eq{
	\ol{J}_{\veps,\tau}(\bmu,\brho) \eqdef \frac{1}{2\tau} \ol{W}_\veps^2(\bmu,\brho) + \ol{F}(\brho)
	\quad\text{where}\quad
	\ol{F}(\brho) \eqdef \sum_{i} v(x_i)\,\brho_i + \ol{u}(\brho_i)
	\quad\tn{and}\quad \ol{u}(s) = u(s/\blebvol) \cdot \blebvol.
}
Note that $u$ acts on the Lebesgue density of $\brho$ and is integrated w.r.t.~the discrete Lebesgue measure.
\subsection{Generalized Sinkhorn Algorithm}

The solution of the problem~\eqref{eq-jko-step-regul-discr} can be conveniently written as $\iteratePar{\brho}{k+1} = \bgamma \ones_\Ndisc$ where $\bgamma$ is the unique solution of 
\eql{\label{eq-kl-form}
	\min_{\bgamma \in \R_+^{N \times N}} \KL( \bgamma|\bK ) + f(\bgamma\ones_\Ndisc) + g(\bgamma^\top\ones_\Ndisc), 
}
where we introduced the Gibbs kernel $\bK_{i,j}=\exp(-c_{i,j}/\veps) \cdot \blebvol^2$, 
where the Kullback-Leibler divergence is
\eq{
	\KL( \bgamma|\bK ) \eqdef \sum_{i,j} \bgamma_{i,j} \log\left( \frac{\bgamma_{i,j}}{\bK_{i,j}} \right) - \bgamma_{i,j} + \bK_{i,j}
}
and the two functionals $f$ and $g$ are
\eql{\label{eq-func-def}
	f(\brho) \eqdef \choice{
		0 \quad\text{if}\quad \brho = \iteratePar{\brho}{k}, \\
		+\infty \quad\text{otherwise}, 
	}
	\quad
	g(\brho) \eqdef \kappa \cdot \ol{F}(\brho)
	\quad\text{with}\quad
	\kappa \eqdef \frac{2\tau}{\veps}.
}
As detailed in~\cite{2016-chizat-sinkhorn}, problems of the form~\eqref{eq-kl-form} have a very strong structure, and can be tackled using highly efficient iterative scaling algorithms. More precisely, the solution of~\eqref{eq-kl-form} can be written in scaling form as $\bgamma_{i,j} = \bK_{i,j} \, \bA_i \, \bB_j$, where the two vectors $(\bA,\bB) \in (\R_+^\Ndisc)^2$ can be computed using the following convergent iterative scheme
\eql{\label{eq-iter-sinkhorn}
	\bA \longleftarrow \frac{\Prox_{f}(\bK \bB)}{\bK \bB}
	\quad\text{and}\quad
	\bB \longleftarrow \frac{\Prox_{g}(\bK^\top \bA)}{\bK^\top \bA}, 
}
where $\tfrac{\cdot}{\cdot}$ denote entry-wise division of vectors. 
Here, we used the so-called proximal operator for the KL divergence, which reads
\eq{
	\Prox_{f}(\bU) \eqdef \argmin_{\bV \in \R_+^\Ndisc} \KL(\bV|\bU) + f(\bV).
}

For the functions~\eqref{eq-func-def} involved for the resolution of the JKO steps, the corresponding proximal operators have a particularly simple form, since
\eq{
	\Prox_{f}(\bU) = \iteratePar{\brho}{k}
	\quad\text{and}\quad
	\Prox_{g}(\bU) = (
		\Prox_{\kappa \ol{u}}(\bU_i \, e^{-\kappa \, v(x_i)}) 
	)_i
}
Note that in this last expression, $\Prox_u$ is the proximal map of a 1-D function, so either it can be computed in closed form or otherwise it can be pre-computed with high accuracy in a look-up-table. 

As explained in~\cite{2016-chizat-sinkhorn}, iterations~\eqref{eq-iter-sinkhorn} are generalization of the celebrated Sinkhorn algorithm~\cite{Sinkhorn64}. When $\veps$ is small, the naive application of $\bK$ leads to instabilities and numerical overflows. We refer to~\cite{2016-chizat-sinkhorn} for details on how to implement these formulas in a numerically stable way.

\subsection{Numerical Simulations}

\begin{table}
	\centering
	\begin{tabular}{lllll}
	\hline
	model & $u(s)$ & $p(s)$ & PDE & $\Prox_{\kappa \ol{u}}(s)$ \\
	\hline
	\hline
	heat & $s\,\log(s)-s$ & $s$ & $\partial_t \rho = \Delta \rho + \ddiv(\rho \nabla v)$ &
		$s^{1/(1+\kappa)}\,\blebvol^{\kappa/(1+\kappa)}$ \\
	\hline
	porous media & $\frac{1}{m-1} s^m$ & $s^m$ & $\partial_t \rho = \Delta \rho^m + \ddiv(\rho \nabla v)$ &
		$\blebvol \left( \frac{\LambertW(m\,(s/\blebvol)^{m-1}\,\kappa)}{m\,\kappa} \right)^{\frac{1}{(m-1)}}$ \\
	\hline
	congestion & $\begin{cases} 0 & \tn{if } s \in [0,1] \\ +\infty & \tn{else} \end{cases}$ &
		--- & --- & 
		$\max\{ \min \{ s,\blebvol\}, 0 \}$ \\
	\hline
	\end{tabular}
	\caption{Overview of numerically studied models and proximal maps. $\LambertW$ denotes the Lambert W function (or product logarithm).}
	\label{table:proxFormulas}
\end{table}

The following simulations are all performed on an equidistant Cartesian grid discretization of $[0,1]^n$ using either $\Ndisc=1024$ points (for $n=1$) or $\Ndisc=256 \times 256$ points (for $n=2$).

Numerically we study the heat equation, corresponding to $u(s)=s \cdot \log(s)-s$ and the porous media equation implied by $u(s)=\frac{1}{m-1} s^m$ for different exponents $m>1$. These are within the scope of Assumption \ref{asp:FreeEnergy}.
In addition, we give a numerical example for $u(s)=0$ if $s \in [0,1]$ and $u(s)=+\infty$ else, which can informally be seen as the $m \to \infty$ limit of the porous media equation or as implementation of a congestion constraint, where the density cannot exceed $1$. While this is not covered by our convergence analysis, the time-discrete JKO scheme is still well-defined and it is instructive to study the behaviour of such flows numerically.
The required formulas related to these models are summarized in Table \ref{table:proxFormulas}.

\begin{figure}
	\centering
	\includegraphics{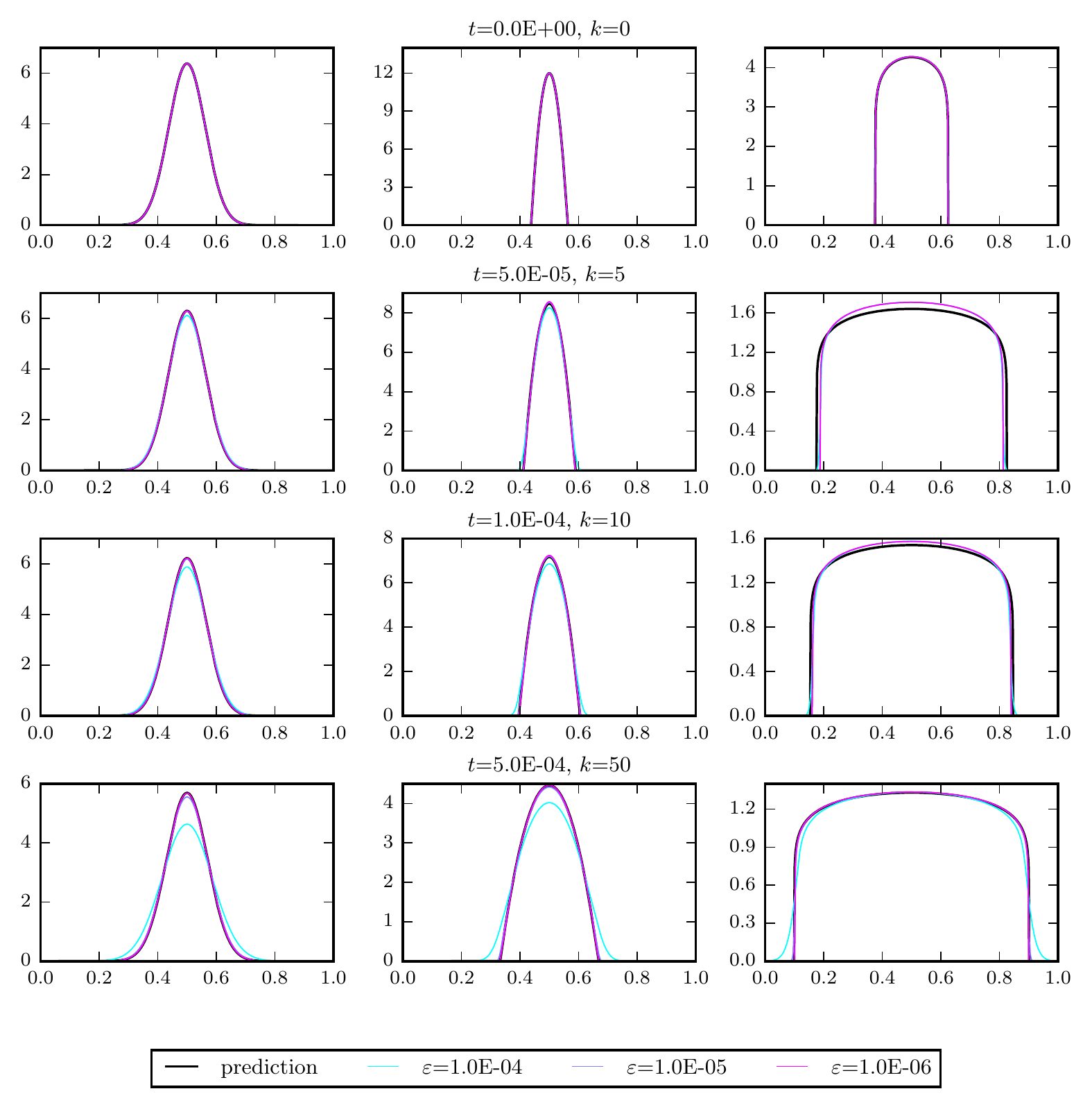}
	\caption{Comparison of different models on: heat equation (left), porous media equation $m=2$ (middle) and $m=10$ (right). The simulated domain is $[0,1]$, time step-size $\tau=10^{-5}$, the $y$-axis shows the (discrete) Lebesgue density of $\rho^{(\veps,\tau,k+1)}$ for various values of $\veps$, rescaled in each tile for better visibility.\\
	As initial profile we chose the fundamental solution of each model (see text) and compare the evolution with the analytic prediction.}
	\label{fig:model_overview}
\end{figure}

A qualitative comparison of different models on $[0,1]$ is given in Fig.~\ref{fig:model_overview}: the evolution of the entropy regularized time-discrete JKO scheme of an initial density profile is compared to the corresponding analytic solution of the PDE.
For the heat equation we consider a Gaussian profile
\begin{align*}
	\rhoSol(t,x)=\frac{1}{\sqrt{4 \pi\,(t+t_0)}} \exp\left(-\frac{(x-x_0)^2}{4\,(t+t_0)}\right)\,.
\end{align*}
For the porous media equation we chose the Barenblatt profile (also known as ZKB solution) \cite{VazquezPorousMedium2006},
\begin{align*}
	\rhoSol(t,x)=(t+t_0)^{-\alpha}\,\left(C-\beta\,(x-x_0)^2\,(t+t_0)^{-2\alpha/n} \right)_+^{\frac{1}{m-1}}
	\quad \tn{with} \quad \alpha=\frac{n}{n\,(m-1)+2}, \quad \beta=\frac{(m-1)\,\alpha}{2\,m\,n}
\end{align*}
and with a suitable normalization constant $C$. We choose some offset $t_0>0$ to avoid the Dirac singularity.
As initial density we set $\rho_0(x)=\rhoSol(0,x)$.

Overall, we observe two types of deviations between the time-discrete numerical scheme and the analytic solution:
For $\veps=10^{-4}$ there is significant blurring in the numerical scheme, introduced by the entropic smoothing. As $\veps$ decreases this effect becomes weaker and is virtually invisible for $\veps=10^{-6}$. In particular the compact support of the Barenblatt profiles is well preserved.
The other effect affects mainly the porous media equation for $m=10$. For small $t$ (or $k$), when the density is still rather concentrated, the pressure is high and the analytic solution changes quickly in time. The time-discrete scheme cannot capture these quick changes with high precision. The agreement becomes better as the solution approaches the steady state.

\begin{figure}
	\centering
	\includegraphics{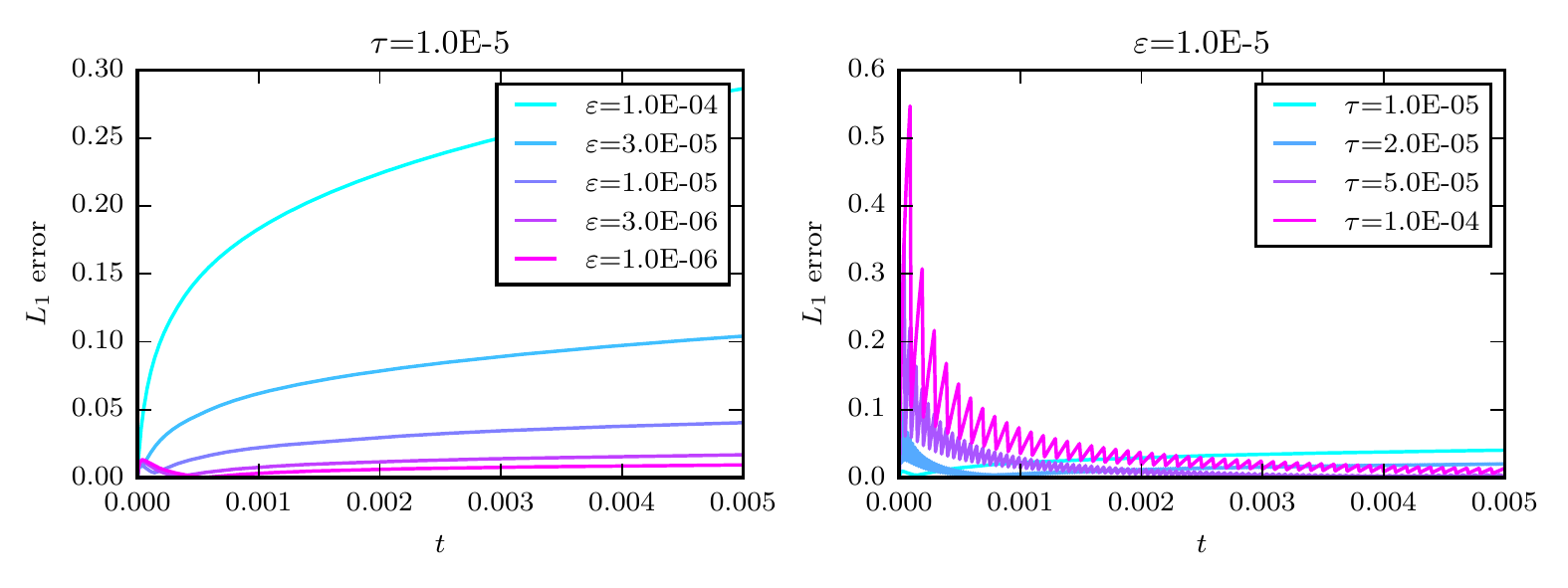}
	\caption{$L_1$ error of a time-slice of time-discrete scheme relative to exact solution, for the porous media equation, $m=2$.
	The $y$-axis shows $\int |\interp(t,x) -\rhoSol(t,x)|\,dx$, for different $\veps$ (left) and different $\tau$ (right).}
	\label{fig:time_error}
\end{figure}

A more quantitative error analysis is given in Figures \ref{fig:time_error} and \ref{fig:L1_error}.
In Fig.~\ref{fig:time_error} the $L_1$ error in space for time slices between the time-discrete and 
predicted solutions is shown for various regularization strengths $\veps$ and time step-sizes $\tau$.
As expected, for fixed $\tau$ the error decreases as $\veps \to 0$ and typically increases with $t$.
For fixed $\veps$ we observe various phenomena:
As $\tau$ increases, the error curves become zigzagged, since the time-discrete interpolation $\interp$ is constant over increasingly longer periods.
For small $t$ the error grows with $\tau$, since the scheme is increasingly unable to resolve quick changes in time (see also Fig.~\ref{fig:model_overview} and related text). As $t$ increases and the changes in time become slower, this source of error decreases (this also explains the initial decrease for small $\veps$ in the left panel).
Conversely, for larger $t$ the primary cause of error is the blur introduced by entropy smoothing. Therefore the error eventually decreases with $\tau$, as fewer time-steps are required, leading to less blur.

\begin{figure}
	\centering
	\includegraphics{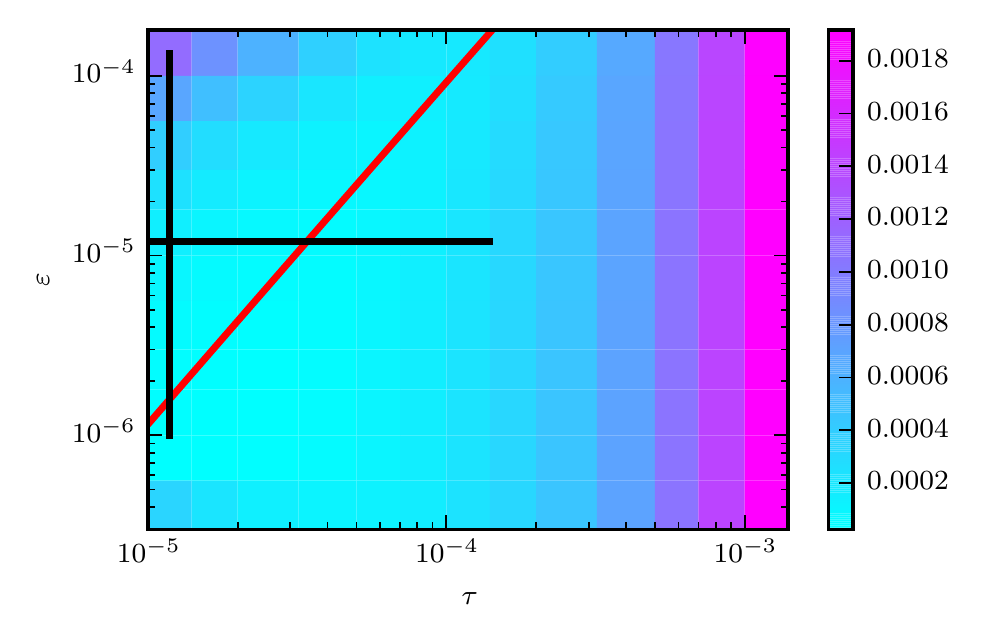}
	\caption{Total $L_1$ error over space and time of time-discrete scheme relative to exact solution, for the porous media equation, $m=2$, for different $\veps$ and different $\tau$.
	The black curves indicate the parameter ranges shown in Fig.~\ref{fig:time_error}. The red curve shows the relation $\veps = C \cdot \tau^2 \cdot |\log \tau|$, for $C=10^3$. This is the sufficient upper bound for convergence of asymptotic scaling between $\tau$ and $\veps$, as given in Sect.~\ref{sec:JKOConvergencePDE}.}
	\label{fig:L1_error}
\end{figure}

Fig.~\ref{fig:L1_error} illustrates the complete spatio-temporal $L_1$ error between the time-discrete scheme and the analytic solution for various combinations of $\veps$ and $\tau$, as well as a curve indicating the asymptotic relation between $\veps$ and $\tau$, proven sufficient for convergence in Sect.~\ref{sec:JKOConvergencePDE}.
While this bound may not be tight, clearly, $\veps$ must decrease sufficiently fast as $\tau \to 0$ for convergence. Otherwise, the blur introduced at each time-discrete step will introduce too much error.
The artifacts at the bottom left corner are due to the spatially discretized grid: for very small $\tau$ and $\veps$ the transport cost to the next grid point can be prohibitive and the gradient flow scheme `freezes'. These artifacts disappear when the scheme is run at a higher spatial resolution.

\begin{figure}
	\centering
	\includegraphics{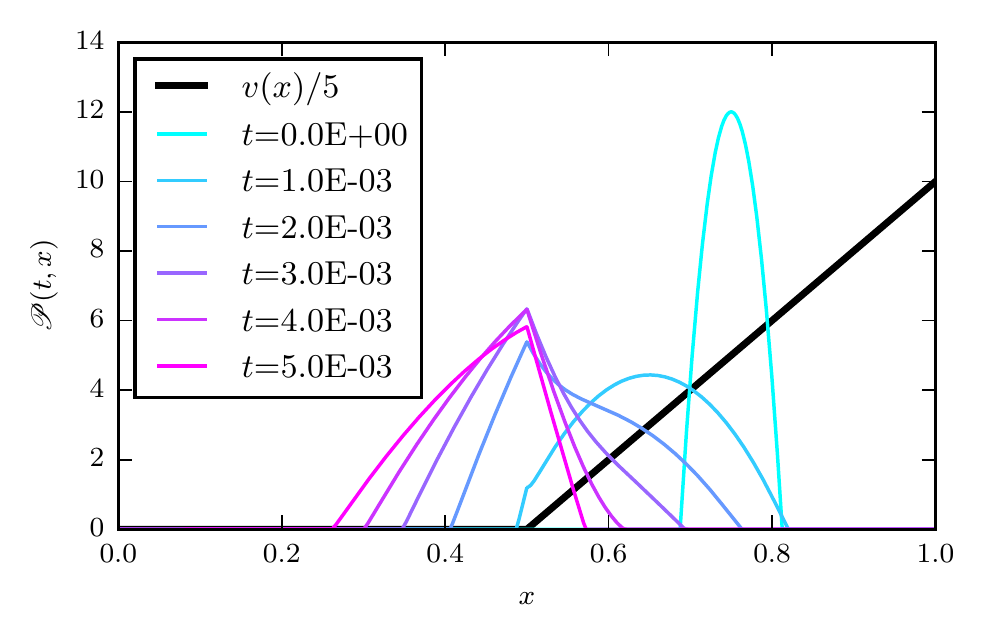}
	\caption{Porous media equation,$m=2$ with drift: $v(x)=100 \cdot \max\{0,x-0.5\}$, $\tau=10^{-5}$, $\veps=10^{-6}$.}
	\label{fig:porous_media_drift}
\end{figure}

An example of the porous media equation, $m=2$, with drift (i.e.~non-zero potential) is displayed in Fig.~\ref{fig:porous_media_drift}. The initial density first `slides down' the slope, there is some `jamming' at the kink and eventually pressure spreads the density out again. Note again, how the compact support is preserved numerically, due to small entropy regularization.

\begin{figure}
	\centering %
	\includegraphics{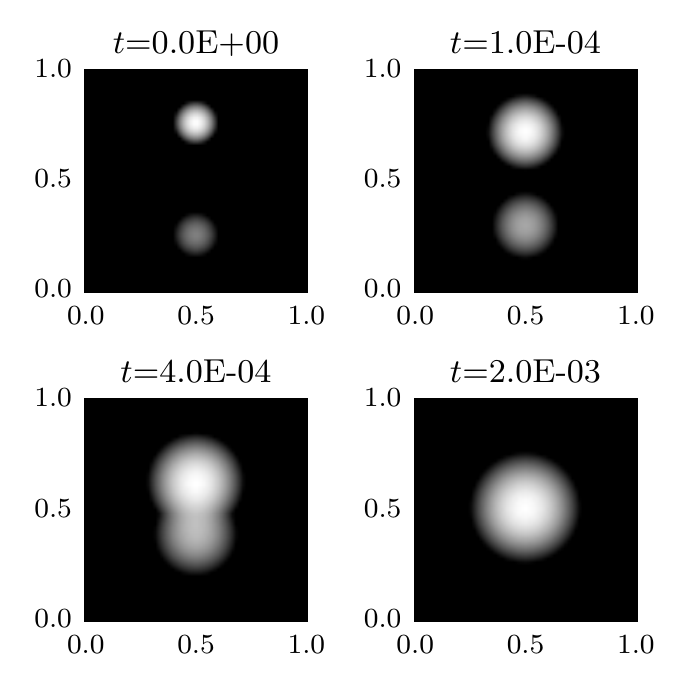} %
	\hfill %
	\includegraphics{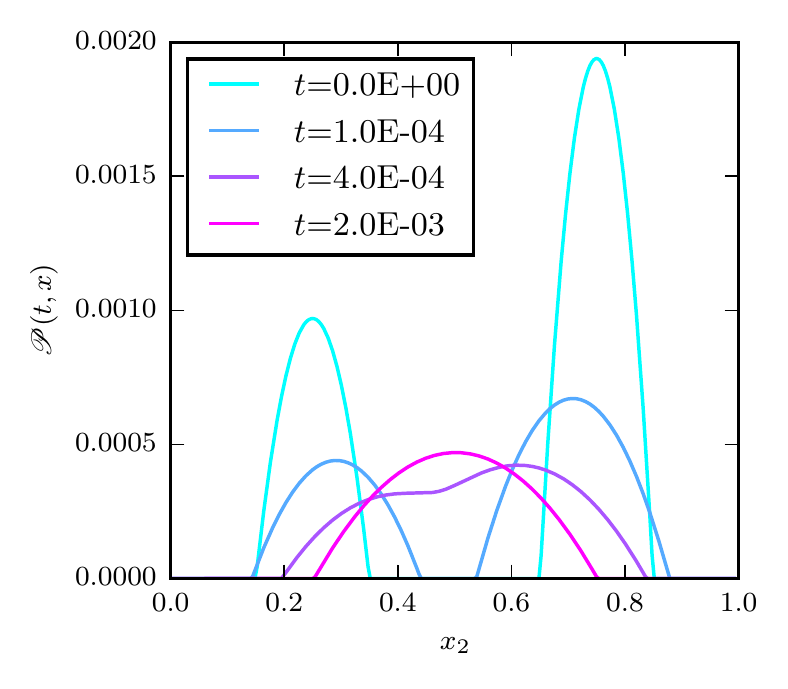} %
	\caption{Porous media equation, $m=2$, on $[0,1]^2$ in a potential $v(x)=10^3 \cdot \|x-(\frac12,\frac12)^\top\|^2$, $\tau=10^{-4}$, $\veps=2.5 \cdot 10^{-5}$. %
	Left: 2-D densities for various times. Separate gray scale for each image, for better visibility. %
	Right: Cross section of densities along the line $x_1=0.5$. %
	}
	\label{fig:2d_dens}
\end{figure}

An example in two dimensions, $n=2$, is shown in Fig.~\ref{fig:2d_dens}. Two initial blobs collide in the potential well and relax into a steady state. Note how the upper blob (with larger amplitude) expands faster, due to higher pressure.

\begin{figure}
	\centering
	\includegraphics{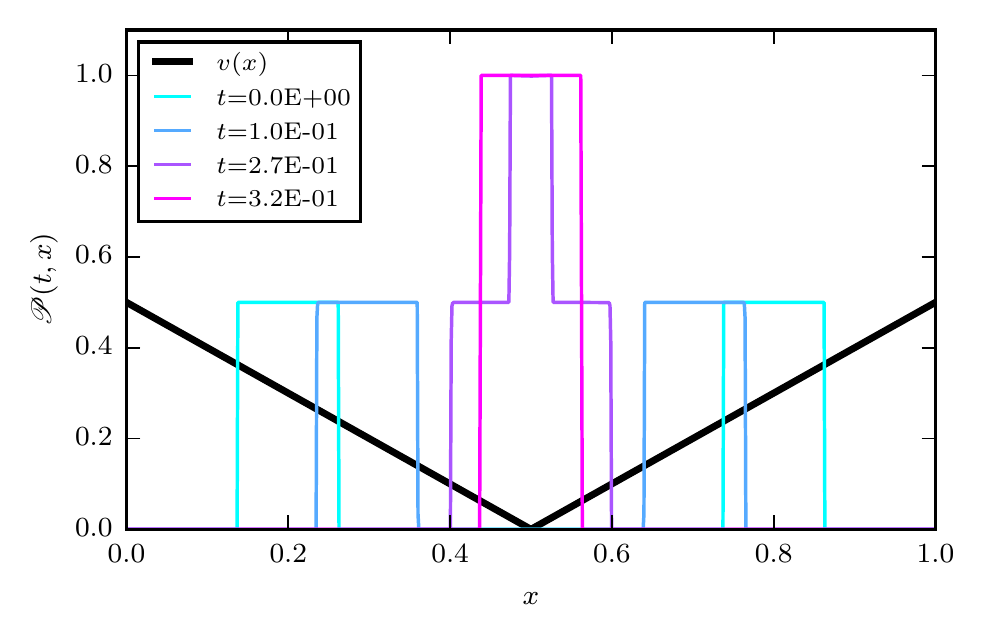}
	\caption{Time-discrete evolution for the congestion model: $u(s)=0$ if $s \in [0,1]$, $u(s)= \infty$ else. Potential $v$ as depicted. $\tau=10^{-2}$, $\veps=10^{-7}$.}
	\label{fig:congestion}
\end{figure}


Finally, we provide a numerical example of the congestion model (Fig.~\ref{fig:congestion}). Two initial blocks with densities below the threshold collide in a potential well. The congestion constraint prevents the density from collapsing into a Dirac mass.


\section*{Conclusion}
\label{sec-conclusion}

In this paper, we have presented a convergence analysis for both the Schr\"odinger problem (i.e. the entropic smoothing of optimal transport) and the entropic approximation of JKO flows. 
We showed that in the limit of a small regularization, one recovers respectively the usual solution to the Kantorovich problem and the gradient flow according to the Wasserstein metric (if the step size in time does not decay too fast with respect to the smoothing). 
The convergence of the entropic JKO scheme was illustrated with numerical examples.

\section*{Acknowledgements}

Guillaume Carlier gratefully acknowledges the support of the ANR, through the project ISOTACE (ANR-12-MONU-0013) and INRIA through the \'equipe-projet MOKAPLAN.
The work of Gabriel Peyr\'e has been supported by the European Research Council (ERC project SIGMA-Vision).
Bernhard Schmitzer is supported by a public grant overseen by the French National Research Agency (ANR) as part of the ``Investissements d'avenir'', program-reference ANR-10-LABX-0098. 
Vincent Duval gratefully acknowledges support from the CNRS (D\'efi Imag'in de la Mission pour l'Interdisciplinarit\'e, project CAVALIERI).

\appendix


\section{Useful Results}

\begin{lemma}
	\thlabel{thm:muAlphaBound}
	Let $\alpha \in (\frac{n}{n+2},1)$, (i.e. $\frac{\alpha}{1-\alpha}>\frac{n}{2}$), there is a finite constant $C>0$ (depending only on $n$ and $\alpha$) such that for all $\mu \in \probac(\R^n)$, $R \geq 0$ one has
	\begin{align*}
		\int_{\R^n \setminus B_R} (\mu(x))^\alpha\,dx \leq (1+M_{\R^n}(\mu))^\alpha
			\cdot C \cdot \left(\frac{1}{R^2+1}\right)^{\alpha-\frac{n}{2}(1-\alpha)}\,.
	\end{align*}
	Note in particular that this integral remains finite for $R=0$ and vanishes in the limit $R \rightarrow \infty$.
\end{lemma}
\begin{proof}
	We repeat the arguments in \cite[Prop.~4.1]{JKO1998}. By H\"older's inequality, we have:
	\begin{align*}
		\int_{\R^n\setminus B_R} (\mu(x))^\alpha\,dx &= \int_{\R^n\setminus B_R} (\mu(x))^\alpha (1+\abs{x}^2)^\alpha \frac{1}{(1+\abs{x}^2)^\alpha }dx\\
		& \leq \left(\int_{\R^n\setminus B_R} \mu(x)\, (1+\abs{x}^2)\,dx\right)^\alpha
			\left(\int_{\R^n\setminus B_R}
			\left(\frac{1}{1+\abs{x}^2} \right)^{\frac{\alpha}{1-\alpha}} dx \right)^{1-\alpha}\\
		& = (1+M_{\R^n}(\mu))^\alpha \left(\int_{\R^n\setminus B_R}
			\left(\frac{1}{1+\abs{x}^2} \right)^{\frac{\alpha}{1-\alpha}} dx \right)^{1-\alpha}\\
		&  \leq (1+M_{\R^n}(\mu))^\alpha \cdot C \cdot \left(\frac{1}{1+R^2}\right)^{\alpha-\frac{n}{2}(1-\alpha)} \qedhere
	\end{align*}
\end{proof}

\begin{corollary}[Entropy Bound, \protect{\cite[Prop.~4.1]{JKO1998}}]
\thlabel{thm:EntropyMomentBound}
There exists a constant $C>0$ and an exponent $0 < \alpha < 1$ such that for every $\mu \in \probac(\R^n)$ one has
	\begin{align}
		H_{\R^n}(\mu) & \geq -C\,(M_{\R^n}(\mu)+1)^\alpha\,, & \HRkneg{n}(\mu) & \leq C\,(M_{\R^n}(\mu)+1)^\alpha\,.
		\label{eq:EntropyMomentBound}
	\end{align}
\end{corollary}
\begin{proof}
	This follows directly from \thref{thm:muAlphaBound} and the bound $z\,\log(z)>-C \cdot z^\alpha$ for $z \geq 0$ for any $\alpha \in (0,1)$ and an appropriate $C>0$.
\end{proof}

\begin{theorem}[\protect{\cite[Thm.~2.34]{AFP00}}]
\thlabel{thm:ambroLSC}
Let $u : [0,\infty) \rightarrow [0,\infty]$ be convex, lower semi-continuous and super-linear: $\lim_{r \rightarrow \infty} u(r)/r = \infty$.
Let $(\nu_k)_{k \in \N}$ be a sequence in $\probacAll(\R^n)$, narrowly converging to some $\nu \in \probAll(\R^n)$. Consider $U: \probAll(\R^n) \rightarrow \R$, given by
\begin{align*}
	U(\mu) = \begin{cases}
		\int_{\R^n} u(\mu(x))\,dx & \tn{if } \mu \in \probacAll(\R^n)\,, \\
		+ \infty & \tn{otherwise.}
		\end{cases}
\end{align*}
Then
\begin{align*}
	U(\nu) \leq \liminf_{k \rightarrow +\infty} U(\nu_k)\,.
\end{align*}
\end{theorem}
\begin{proof}
	This is a direct application of \cite[Thm.~2.34]{AFP00}, taking the Lebesgue measure as reference measure, using that Borel probability measures on $\R^n$ are Radon measures and that narrow convergence implies local weak$^\ast$ convergence in the sense of \cite{AFP00}.
\end{proof}

We now provide an extension to a somewhat larger class of integral functions $u$:
\begin{corollary}
  \thlabel{cor:fctLSC}
  Let $u:[0,\infty) \rightarrow (-\infty,\infty]$ be a convex, lower semi-continuous function with super-linear growth at infinity, such that $u(0)=0$ and there exist constants $C>0$,  $\alpha\in(\frac{n}{n+2} ,1)$ with 
\begin{align}
	\label{eq:coroFLBound}
	u(s) \geq -C\,s^\alpha \quad \tn{for all } s \in [0,\infty)\,.
\end{align}
Let $U$ be built from $u$ as in \thref{thm:ambroLSC}. Let $(\nu_k)_{k \in \N}$ be a sequence in $\prob(\R^n)$, narrowly converging to some $\nu \in \prob(\R^n)$ and assume there is some $C_M<\infty$ such that $M(\nu_k)$, $M(\nu) < C_M$.

\noindent Then,
\begin{align*}
  U(\nu)\leq \liminf_{k\to+\infty} U(\nu_k).
\end{align*}

\end{corollary}

\begin{proof}
	We write $u^+=\max\{u,0\}$, $u^-=\max\{-u,0\}$ so that $u=u^+-u^-$. In particular, $0\leq u^-(s) \leq C\,s^\alpha$.
	
	Then, for some $R>0$, $\mu \in \probac(\R^n)$ we define
	\begin{align*}
		U|_{B_R}(\mu) & = \int_{B_R} u(\mu(x))\,dx\,,\\
		U^+|_{\R^n \setminus B_R}(\mu) & = \int_{\R^n \setminus B_R} u^+(\mu(x))\,dx\,,\\
		U^-|_{\R^n \setminus B_R}(\mu) & = \int_{\R^n \setminus B_R} u^-(\mu(x))\,dx\,.
	\end{align*}
	and $U|_{B_R}(\mu)=\infty$, $U^+|_{\R^n \setminus B_R}(\mu)=\infty$ when $\mu$ is singular w.r.t.~the Lebesgue measure. For $U^-|_{\R^n \setminus B_R}(\mu)$ one would simply integrate over the non-singular part of $\mu$.
	
	From \thref{thm:muAlphaBound} we conclude $0 \leq U^-|_{\R^n \setminus B_R}(\mu) \leq (1 + M(\mu))^\alpha \cdot g(R)$ for a non-negative function $g(R)$ that vanishes in the limit $R \rightarrow \infty$. Thus, for $M(\mu) < \infty$ we have the decomposition
	\begin{align*}
		U(\mu) & = U|_{B_R}(\mu) + U^+|_{\R^n \setminus B_R}(\mu) - U^-|_{\R^n \setminus B_R}(\mu)\,.
	\end{align*}
	\thref{thm:ambroLSC} can be applied to the first two terms. For $U|_{B_R}(\mu)$ note that $u$ is bounded from below and $B_R$ has finite volume. Thus, we can temporarily add a constant to make it non-negative.
	
	We then find for all $0 < R < \infty$:
	\begin{align*}
		\liminf_{k \rightarrow \infty} U(\nu_k) & \geq 
			\liminf_{k \rightarrow \infty} U|_{B_R}(\nu_k) + 
			\liminf_{k \rightarrow \infty} U^+|_{\R^n \setminus B_R}(\nu_k) - (1 + M(\nu_k))^\alpha \cdot g(R)\\
			& \geq U|_{B_R}(\nu) + U^+|_{\R^n \setminus B_R}(\nu) -(1 + C_M)^\alpha \cdot g(R) \\
			& = U(\nu) + U^-|_{\R^n \setminus B_R}(\nu) - (1 + C_M)^\alpha \cdot g(R) \\
		& \geq U(\nu) - (1 + M_C)^\alpha \cdot g(R)
	\end{align*}
	The last term can be made arbitrarily small by increasing $R$, thus the proof is complete.
\end{proof}

\bibliographystyle{siam}
\bibliography{references}

\end{document}